\documentclass[11pt,reqno]{amsart}
\usepackage[utf8]{inputenc}
\usepackage{ytableau}
\usepackage{youngtab}
\usepackage{comment}
\usepackage{mathtools}
\usepackage{tikz-cd}
\usepackage{color}
\usepackage{hyperref}
\usepackage{amsmath}
\usepackage{amssymb}
\usepackage{amsfonts}
\usepackage{graphicx}
\usepackage[all]{xy}
\usepackage{verbatim}
\usepackage{mathdots}
\usepackage[left=2.4cm,right=2.2cm,top=2.6cm,bottom=3.4cm]{geometry}
\hypersetup{colorlinks,linkcolor={blue},citecolor={blue},urlcolor={blue}}
\usepackage[colorinlistoftodos]{todonotes}% colorful comments
\newtheorem{theorem}{Theorem}[section]
\newtheorem{lemma}[theorem]{Lemma}
\newtheorem{corollary}[theorem]{Corollary}
\newtheorem{proposition}[theorem]{Proposition}
\newtheorem*{theorem-nonumber}{Theorem}
\newtheorem*{definition-nonumber}{Definition}
\newtheorem*{proposition-nonumber}{Proposition}

\theoremstyle{definition}
\newtheorem{definition}[theorem]{Definition}
\newtheorem{example}[theorem]{Example}
\newtheorem{remark}[theorem]{Remark}

\newcommand{\g}{\mathfrak{g}}

\newcommand{\liep}{\mathfrak{p}}
\newcommand{\bor}{\mathfrak{b}}
\newcommand{\syp}{\mathfrak{sp}_{2n}}
\newcommand{\sy}{\mathfrak{sp}}
\newcommand{\n}{\mathfrak{n}}
\newcommand{\h}{\mathfrak{h}}
\newcommand{\J}{\mathsf{J}}
\newcommand{\hv}{\nu_{\lambda}}
\newcommand{\U}{\mathcal{U}}
\newcommand{\V}{\textsc{V}}
\newcommand{\Sm}{\mathcal{S}}
\newcommand{\Poly}{\textsc{P}}
\newcommand{\N}{\mathcal{N}}
\newcommand{\F}{\mathcal{F}_{\lambda}}
\newcommand{\W}{\textsc{W}}
\newcommand{\X}{\textsc{X}}
\newcommand{\Lm}{\textsc{L}}
\newcommand{\Jm}{\textsc{J}}

\newcommand{\grass}{\textsc{Sp}\textsc{G}\mathsf{r}}
\newcommand{\Sp}{\textsc{Sp}}
\newcommand{\group}{\textsc{Sp}_{2n}}
\newcommand{\parabolic}{\textsc{P}}
\newcommand{\G}{\textsc{G}}

\newcommand{\vectorspace}{\textsc{U}}
\newcommand{\Gr}{\textsc{G}\textsf{r}}
\newcommand{\SL}{\textsc{SL}_{n}}

\newcommand{\variety}{\mathcal{V}}
\newcommand{\M}{\textsc{M}}
\newcommand{\I}{\textsc{I}}
\newcommand{\ideal}{\mathcal{I}}
\newcommand{\T}{\textsc{T}}
\newcommand{\set}{\textsc{S}}
\newcommand{\Ideal}{\mathfrak{I}}
\newcommand{\Tset}{\textsf{T}}
\newcommand{\borel}{\textsc{B}}
\newcommand{\complete}{\textsc{Sp}\mathcal{F}_{2n}}
\newcommand{\stiefel}{\mathcal{W}}
\newcommand{\w}{\textsf{w}}
\newcommand{\comp}{\textsc{Sp}\mathcal{F}}
\newcommand{\SyST}{\textsc{SyST}}
\newcommand{\p}{\textsf{p}}
\newcommand{\dk}{\textsf{d}}

\begin{document}
\thispagestyle{empty}
\title[Symplectic PBW degenerate flag varieties]{Symplectic PBW degenerate flag varieties; PBW tableaux and defining equations}
\author{George Balla}
\address{Algebra and Representation Theory, RWTH Aachen University, Pontdriesch 10-16, 52062 Aachen, Germany}
\email{balla@art.rwth-aachen.de}
\maketitle

\vspace{-10.00mm} 
\begin{abstract}
We define a set of PBW-semistandard tableaux that is in a weight preserving bijection with the set of monomials corresponding to integral points in the Feigin-Fourier-Littelmann-Vinberg polytope for highest weight modules of the symplectic Lie algebra. We then show that these tableaux parametrize bases of the multi-homogeneous coordinate rings of the complete symplectic original and PBW degenerate flag varieties. From this construction, we provide explicit degenerate relations that generate the defining ideal of the PBW degenerate variety with respect to the Pl\"ucker embedding. These relations consist of type {\tt A} degenerate Pl\"ucker relations and a set of degenerate linear relations that we obtain from De Concini's linear relations.\\
    
    %\noindent \textbf{Key words:} Symplectic groups $\cdot$ Flag varieties $\cdot$ PBW degeneration $\cdot$ Degenerate relations
    
    %\noindent \textbf{Mathematics Subject Classification 2020:} 17B45 $\cdot$ 05E10 $\cdot$ 14L35
    \end{abstract}

\vspace{-2.00mm}

\section{Introduction}
Let $\G$ be a simple, simply connected algebraic group over the field $\mathbb{C}$ and $\g$ the corresponding Lie algebra. Let $\g=\n^{+}\oplus\h\oplus\n^{-}$ be a Cartan decomposition and $\bor = \n^{+}\oplus\h$ the Borel subalgebra.
For a dominant, integral weight $\lambda$, let $\V_{\lambda}$ be the corresponding simple $\g$-module, and $\hv\in \V_{\lambda}$ a highest weight vector. 
 For $\lambda$ regular, the complete flag variety $\F$ is defined to be the closure of the $\G$-orbit through a highest weight line: $\F=\overline{\G[\hv]} \hookrightarrow{} \mathbb{P}(\V_{\lambda}).$
Another realisation of this variety is through the quotient $\G/{\borel}$, where $\borel$ is a Borel subgroup. 

On the other hand, one has $\V_{\lambda}=\U(\n^{-})\hv$, where $\U(\n^{-})$ is the universal enveloping algebra of $\n^{-}$. 
There exists a degree filtration $\U(\n^{-})_s=\text{span}\{x_1\cdots x_l: \,\, x_i\in \n^{-}, \,\, l\leq s\}$ on $\U(\n^{-})$. 
This filtration in turn induces the filtration $F_s=\U(\n^{-})_s\hv$ on $\V_{\lambda}$, called the \textit{PBW filtration}. 
The associated graded space is $F_0 \oplus_{s\geq 1}F_{s}/F_{s-1}$, which will be denoted by $\V_{\lambda}^a$ (see \cite{FFL11a} and \cite{FFL11b}). 

This graded space has a structure of $\g^a$-module where $\g^a$ is a Lie algebra which is a semi-direct sum of $\bor$ and an abelian ideal $(\n^-)^a$. Let $\G^a$ be a Lie group corresponding to $\g^a$.  Let $\hv^a$ be the image of $\hv$ in $\V_{\lambda}^a$. The PBW degenerate flag variety is defined to be $\F^a := \overline{\G^a [\hv^a]} \hookrightarrow \mathbb{P}(\V_{\lambda}^a)$ (\cite{Fei12}).\\

Feigin in \cite{Fei12}, studied the variety $\F^a$ in type {\tt A} when $\G = \SL(\mathbb{C})$ and $\g = \mathfrak{sl}_n(\mathbb{C})$. The Pl\"ucker embedding of the original variety $\F$ and the PBW degenerate variety $\F^a$ into the product of projective spaces was considered, i.e.,
\[\F\xhookrightarrow{} \prod_{k=1}^{n-1}\mathbb{P}\Big(\bigwedge ^k\mathbb{C}^{n}\Big )\quad\text{and}\quad \F^a\xhookrightarrow{} \prod_{k=1}^{n-1}\mathbb{P}\Big(\bigwedge ^k\mathbb{C}^{n}\Big).\]
In order to show that the variety $\F^a$ is a flat degeneration of the original variety $\F$, he defined the PBW-semistandard tableaux which label bases of the multi-homogeneous coordinate rings of both varieties. 
Let us review what these tableaux are. 
For a type {\tt A}$_{n-1}$ dominant, integral weight $\lambda$, written as a partition $\lambda = (\lambda_1\geq \lambda_2\geq \cdots\geq \lambda_{n-1}\geq 0)$, consider the corresponding Young diagram $Y_{\lambda}$ (English convention).

A type {\tt A} PBW-semistandard tableau of shape $\lambda$ is the filling of $Y_{\lambda}$ with entries from $\{1,\ldots,n\}$ such that the following three conditions are satisfied. First of all, in each column, each entry less than the length of that column is at row position equal to that entry (or in short, at its position). 
Secondly, every entry that is not at its position should be greater than all entries below it in any given column. And finally, for every entry in each column apart from the first column, there should be a greater or equal entry in the column to the left and in the same row or in a row below. We refer to the last condition as PBW-semistandardness.\\

Now consider type {\tt C}, with $\G = \group(\mathbb{C})$ and $\g=\syp(\mathbb{C})$. We consider the complete symplectic flag variety, which will be denoted by $\complete$ and its PBW degeneration, which will be denoted by $\complete^a$. We again consider the Pl\"ucker embedding of these varieties into the product of projective spaces:
\[\complete\xhookrightarrow{} \prod_{k=1}^{n}\mathbb{P}\Big(\bigwedge ^k\mathbb{C}^{2n}\Big )\quad\text{and}\quad \complete^a\xhookrightarrow{} \prod_{k=1}^{n}\mathbb{P}\Big(\bigwedge ^k\mathbb{C}^{2n}\Big).\]
Let $\mathbb{C}[\complete]$ and $\mathbb{C}[\complete^a]$ denote the multi-homogeneous coordinate rings of $\complete$ and $\complete^a$ with respect to the above embeddings.  The first goal of this paper is to define a set of PBW-semistandard tableaux for type {\tt C}$_n$, and to show that it labels weighted bases of both $\mathbb{C}[\complete]$ and $\mathbb{C}[\complete^a]$.
Let $\lambda$ be a type {\tt C}$_n$ dominant, integral weight, written again as a partition $\lambda = (\lambda_1\geq \lambda_2\geq \cdots\geq \lambda_{n}\geq 0)$. For $i\in  \{1,\ldots, n\}$, let $\overline{i}:=2n+1-i$.

A symplectic (or type {\tt C}) PBW-semistandard tableau of shape $\lambda$, is a filling of the corresponding Young diagram $Y_{\lambda}$ with entries in the set $\{1<\cdots<n<\overline{n}<\cdots<\overline{1}\}$ such that not only the conditions for the type {\tt A} PBW-semistandard tableaux are satisfied, but also the following extra condition. For every element $i\in \{1,\ldots,n\}$ in any column, if the element $\overline{i}$ exists in the same column, then the position of $\overline{i}$ should be above that of $i$, whenever $i$ is less than the length of the column. We would like to note that several nice  symplectic tableaux already exist, for example those of De Concini \cite{Dec79}, Hamel and King \cite{HK11}, Kashiwara and Nakashima \cite{KN91}, King \cite{Kin76}, and Proctor \cite{Pro90}. The main difference between these tableaux and those defined here is the PBW-semistandardness condition (see Subsection \ref{subsec:comparison} for a brief comparison). We prove
\begin{theorem} [Theorem \ref{thm:final-result-two}]
The symplectic PBW-semistandard tableaux index a basis of $\mathbb{C}[\complete^a]$.
\end{theorem}

Feigin, Finkelberg and Littelmann showed in \cite{FFL14} that $\complete^a$ is a flat degeneration of $\complete$. 
It therefore follows naturally that the symplectic PBW-semistandard tableaux also label a basis for $\mathbb{C}[\complete]$ (see Proposition \ref{thm:final-result}). 
In the light of their combinatorics, we would like to discuss a correspondence between these tableaux and certain bases of the modules $\V_{\lambda}$ and $\V_{\lambda}^a$. 
In \cite{FFL11a} and \cite{FFL11b}, Feigin, Fourier and Littelmann defined the Feigin-Fourier-Littelmann-Vinberg polytopes that parametrize monomial bases for highest weight original and PBW degenerate simple modules for a Lie algebra $\g$ in types {\tt A} and {\tt C} respectively. 
Bases arising this way are called FFLV bases. Their existence was first conjectured by Vinberg (see \cite{V05}), who also proved the conjecture for $\mathfrak{sl}_4$, $\sy_4$ and {\tt G}$_2$. We prove that one has a weight preserving bijection between the FFLV basis for the symplectic modules $\V_{\lambda}$ and $\V_{\lambda}^a$ and the symplectic PBW-semistandard tableaux (see Theorem \ref{thm:correspondence}). 

It is worth noting that Young \cite{Yng28} was the first to introduce (semi-)standard Young tableaux to provide a basis for the irreducible polynomial representations of the general linear group and for the irreducible representations of symmetric groups. 
On the other hand, standard monomial theory was begun by Hodge \cite{Hge43}, who used Young theory to give a basis for homogeneous coordinate rings of Grassmann and flag varieties. The same theory has been tremendously developed through the work of different authors (\cite{Dec79}, \cite{LMS79}, \cite{LS91}, \cite{Lmn98}, $\ldots$). \\

At this point, we would like to step back and discuss briefly one of the very important tools in the proof of Theorem \ref{thm:final-result-two}; namely, the degenerate relations. 
Feigin in \cite{Fei12} defined the degenerate Pl\"ucker relations and proved that they generate the defining ideal of the PBW degenerate flag variety in type {\tt A}. Since $\complete^a$ is point-wise contained in the type {\tt A}$_{2n-1}$ complete PBW degenerate flag variety (\cite{FFL14}), it follows that Feigin's degenerate relations are also satisfied on $\complete^a$. We denote these relations by $R_{\Lm,\Jm}^{t;a}$.

On the other hand,  De Concini \cite{Dec79} defined certain linear relations while showing that his symplectic standard tableaux index a basis for $\mathbb{C}[\complete]$. We call these \emph{symplectic relations}, which will be denoted by $S_{(\I_2,\I_1)}$. In his proof, he also used Pl\"ucker relations, which implies that these and the symplectic relations generate the defining ideal of $\complete$, since they provide a straightening law for $\mathbb{C}[\complete]$. Note that Chiriv\`i and Maffei in \cite{CM14} and in \cite{CML09} with Littelmann, gave a general framework for these defining equations for flag varieties and other spherical varieties. We now obtain degenerate relations from the symplectic relations, which we call \emph{symplectic degenerate relations} and denote them by $S_{(\I_2,\I_1)}^a$ (see Definition \ref{def:symplectic-degenerate-relations} for a full description). 

We obtain a fundamental result about the defining ideal of $\complete^a$, which is the second and final goal of this paper. 
Let $\Ideal^a$ be the ideal generated by the relations $S_{(\I_2,\I_1)}^a$ and $R_{\Lm,\Jm}^{t;a}$. 
For example, for $n=2$, the ideal $\Ideal^a$ is generated by the relations:

\vspace*{-4mm}
\begin{align*}
&R_{(1,2),(\overline{2},\overline{1})}^{1;a}
:=\X_{1,2}^a\X_{\overline{2},\overline{1}}^a-\X_{1,\overline{2}}^a\X_{2,\overline{1}}^a+\X_{1,\overline{1}}^a\X_{2,\overline{2}}^a, \,\, R_{(1,2),(\overline{1})}^{1;a} :=\X_{1,2}^a\X_{\overline{1}}^a+\X_{2,\overline{1}}^a\X_1^a,\\
&R_{(1,\overline{2}),(\overline{1})}^{1;a} :=\X_{1,\overline{2}}^a\X_{\overline{1}}^a+\X_{\overline{2},\overline{1}}^a\X_1^a-\X_{1,\overline{1}}^a\X_{\overline{2}}^a, \qquad\,\,\,\, R_{(2,\overline{2}),(\overline{1})}^{1;a}
:=\X_{2,\overline{2}}^a\X_{\overline{1}}^a-\X_{2,\overline{1}}^a\X_{\overline{2}}^a,\\
&R_{(1,2),(\overline{2})}^{1;a} :=\X_{1,2}^a\X_{\overline{2}}^a+\X_{2,\overline{2}}^a\X_1^a, \qquad\qquad\qquad\quad\,\,  \text{and}\quad
S_{(1,\overline{1})}^a:=\X_{1,\overline{1}}^a+\X_{2,\overline{2}}^a.
\end{align*}
We prove:
\begin{theorem}[Theorem \ref{thm:defining-relations-degenerate}]
The ideal $\Ideal^a$ is the prime defining ideal of the variety $\complete^a$ under the Pl\"ucker embedding, $\complete^a\xhookrightarrow{} \prod_{k=1}^n\mathbb{P}\Big(\bigwedge^k \mathbb{C}^{2n}\Big).$
\end{theorem}

In the framework of PBW degenerations, the root vectors of the Lie algebra $\n^-$ are each assigned degree one. It is interesting to note that one can systematically assign different weights to these root vectors such that the associated graded space of $\n^-$ still naturally admits a non-trivial graded Lie algebra structure. To such a structure, one then associates an appropriate Lie group structure. Weighted analogues of the usual PBW degenerate flag varieties arise this way. 

This construction was carried out by Fang, Feigin, Fourier and Makhlin in \cite{FFFM19} for type {\tt A}. They used the combinatorics of PBW-semistandard tableaux to show that these varieties are well behaved, in the sense that they are irreducible. In the sequel, they constructed an explicit maximal prime cone of the tropical flag variety coming from the system of weights assigned to the root vectors of $\n^-$, hence obtaining yet another phenomenal connection of Lie theory to tropical geometry. Moreover, they showed that every point in the relative interior of this cone corresponds to the FFLV toric degeneration. Furthermore, they identified several facets corresponding to linear degenerations (\cite{GFFFR19}).

In a forthcoming work, we will carry out similar constructions for $\syp$ and hence establish connections to the tropical symplectic flag variety (\cite{BO21}).
In the same spirit, we are also extending the work of Bossinger, Lambogila, Mincheva and Mohammadi \cite{BLMM17} on computing toric degenerations arising from tropicalization of flag varieties to the symplectic set-up.\\

This paper is organised as follows: In Section \ref{sec:preliminaries}, we recall results on the FFLV basis for the symplectic Lie algebra. In Section \ref{sec:tableaux}, we define the symplectic PBW-semistandard tableaux and establish the bijection between them and the symplectic FFLV basis. We then show that these tableaux label a basis for the multi-homogeneous coordinate ring of $\complete$ in Section \ref{sec:three}. 
In Section \ref{sec:four}, we give the definition of the symplectic degenerate relations and use them together with the degenerate Pl\"ucker relations to show that the symplectic PBW-semistandard tableaux label a basis for the coordinate ring of $\complete^a$. We also prove here that the ideal generated by these relations is the defining ideal of $\complete^a$ under the Pl\"ucker embedding.

\subsection*{Acknowledgements}
 The author would like to extend his gratitude to his doctoral advisor, Ghislain Fourier, for many useful and insightful discussions on this work and its extensions. Many thanks to anonymous referees for key comments and suggestions that improved on the presentation of this paper. Similarly, great thanks to Xin Fang, Evgeny Feigin, Johannes Flake, Peter Littelmann, and Jorge Alberto Olarte, for several important expositions on these subjects. The author also extends his gratitude to Johannes Flake for technical support with the computer codes that verified our results and to Xin Fang for reading the first version of this paper. The author is funded by the Deutscher Akademischer Austauschdienst (DAAD, German Academic Exchange Service) scholarship: Research Grants - Doctoral Programmes in Germany (Programme-ID 57440921). This work is a contribution to  Project-ID 286237555 - TRR 195 - - by the Deutsche Forschungsgemeinschaft (DFG, German Research Foundation).

\numberwithin{equation}{section}
\section{Preliminaries; Representation Theory}\label{sec:preliminaries}
In this section, we recall the description of the corresponding simple original and PBW degenerate modules for the symplectic Lie algebra and the FFLV basis as studied in \cite{FFL11b}.

\subsection{The Symplectic Lie Algebra; a Brief Description} All information in this subsection can be found in \cite{FH13}.
 Let $\g = \syp(\mathbb{C}) = \n^{+}\oplus  \h \oplus \n^{-}$ be a Cartan decomposition and $\bor = \n^{+}\oplus  \h$ the Borel subalgebra. Let $\Lambda^+$ denote the set of dominant integral weights of $\syp$, and let $\{\omega_1,\ldots, \omega_n\}$ be the set of fundamental weights. Let $\{\varepsilon_1,\ldots,\varepsilon_n\}$ be the standard basis for $\h^*$ with respect to the killing form. Let $\alpha_1,\ldots, \alpha_n$ be the simple roots of $\syp$ and let $\Phi^{+}$ denote the set of positive roots. For $1\leq j\leq n$, let $\overline{j} :=2n+1-j$. The set $\Phi^{+}$ is the union of the following two sets of roots:
\begin{align}\label{eqn:roots}
\alpha_{i,j} &= \alpha_i + \alpha_{i+1}+\ldots+\alpha_{j},\,\, \qquad\qquad\qquad\qquad 1\leq i\leq j\leq n,\nonumber\\
\alpha_{i,\overline{j}} &= \alpha_i + \alpha_{i+1}+\ldots+\alpha_{n}+\alpha_{n-1}+\ldots+\alpha_{j},\,\, 1\leq i\leq j \leq n,
\end{align}
where $\alpha_{i,n}=\alpha_{i,\overline{n}}$. For each $\alpha \in \Phi^+$, fix a non zero element $f_{\alpha}\in \n^{-}_{-\alpha}$, where $\n^{-}_{-\alpha}$ is the set of root vectors in $\n^{-}$ of weight $-\alpha$. Henceforth, we will sometimes use the short forms: $$\alpha_{\overline{i}}= \alpha_{i,\overline{i}},\quad f_{i,j}=f_{\alpha_{i,j}} \quad \text{and}\quad  f_{i,\overline{j}}=f_{\alpha_{i,\overline{j}}}.$$ 

Let $\W$ be a $2n$-dimensional vector space over $\mathbb{C}$ and let $\{\w_1,\ldots,\w_{2n}\}$ be its fixed basis. We fix a non-degenerate symplectic form $\langle\,\,,\,\,\rangle$, defined by:\\
$\langle \w_i \,\,,\w_{\overline{i}}\,\,\rangle=1\quad \text{for} \quad 1\leq i\leq n \quad \text{and}$ 
$\quad \langle \w_i \,\,,\w_j\,\,\rangle=0\quad \text{for all} \quad 1\leq i,j \leq n, j\neq \overline{i}$.\\
The symplectic group $\group(\mathbb{C})$ can be realized as the group of automorphisms of $\W$ leaving the form $\langle\,\,,\,\,\rangle$ invariant. Consider a maximal torus $\mathfrak{T}\subset \group$ consisting of diagonal matrices $\mathfrak{t}$ given as follows:
\[\mathfrak{T}=\left\{\mathfrak{t}=
\begin{pmatrix}
t_1 &   &      &         &\\
&   t_2 &      &         &\\
&   &   \ddots &         &\\
&   &   &      t_2^{-1}  &\\
&   &   &      &         t_1^{-1}
          
\end{pmatrix} \,\,\Bigg |\,\, t_1,\ldots, t_n \in \mathbb{C}^*\right\}, \] and a Borel subgroup $\borel\subset \group$ of upper triangular matrices. With respect to this realization, explicit formulas for root vectors of the symplectic Lie algebra $\syp$ are given below:
$$
\begin{array}{rcll}
f_{i,\overline{i}} &=& E_{\overline{i},i}, & \text{ for } 1 \leq i \leq n,\\
f_{i,j} & = & E_{j+1,i}-E_{\overline{i},\overline{j+1}},& \text{ for } 1\leq i\leq j <n,\\
f_{i,\overline{j}} & = & E_{\overline{j},i}+E_{\overline{i},j}, & \text{ for } 1\leq i < j\leq n,
\end{array}
$$
where $E_{p,q}$ is the matrix  with zeros everywhere except for the entry 1 in the $p$-th row and $q$-th column.

\subsection{The Poincar\'e-Birkhoff-Witt (PBW) Degeneration}
Consider the increasing degree filtration on the universal enveloping algebra, $\U(\n^{-})$:
\begin{equation}\label{eqn:filtration}
    \U(\n^{-})_s=\mathsf{span}\{x_1\cdots x_l: x_i\in \n^{-}, l\leq s\}.
\end{equation}

 For a dominant integral weight $\lambda = m_1\omega_1+\ldots+m_n \omega_n\in\Lambda^+$, let as usual, $\V_{\lambda}$ be the corresponding simple highest weight $\syp$-module with a highest weight vector $\hv$. It is known that $\V_{\lambda}=\U(\n^{-})\hv$, therefore, the filtration \eqref{eqn:filtration} induces an increasing degree filtration $F_s$ on $\V_{\lambda}$:
$$F_s=\U(\n^{-})_s\hv.$$
This filtration is called the \emph{PBW filtration}. Let us denote the associated graded space by $\V_{\lambda}^a$. One has:
$$\V_{\lambda}^a = \bigoplus_{s\geq 0} \V_{\lambda}^a(s) =  \bigoplus_{s\geq 0}F_s/F_{s-1}.$$ Elements of $\V_{\lambda}^a(s)$ are said to be homogeneous of \emph{PBW-degree} $s$. The graded space $\V_{\lambda}^a$ has a structure of $\g^a$-module, where $\g^a$ is a semi-direct sum of the Borel subalgebra $\bor$ and an abelian ideal $(\n^-)^a$, which is isomorphic to $\n^-$ as a vector space. The Lie algebra $\g^a$ is said to be the \emph{PBW degeneration} of $\g$ (see \cite{Fei12}). For the highest weight vector $\hv$ in $\V_{\lambda}$, we denote by $\hv^a$ its image in $\V_{\lambda}^a$.

\subsection{The Symplectic FFLV Basis}
Here we recall results due to Feigin, Fourier and Littelmann in \cite{FFL11b}. Our results on the symplectic PBW-semistandard tableaux strongly rely on these results. We first recall the notion of the symplectic Dyck path. Let $\J$ denote the set $\{1,\ldots,n,\overline{n-1},\ldots,\overline{1}\}$ with the usual order: $1<\cdots<n<\overline{n-1}<\cdots<\overline{1}$.

\begin{definition}
A \emph{symplectic Dyck path} is a sequence $\dk=(\dk(0),\ldots,\dk(k))$, $k\geq 0$, of positive roots satisfying the conditions:
\begin{enumerate}
    \item[(i)] the first root $\dk(0)= \alpha_i$ for some $1\leq i \leq n$, i.e., it is simple.
    \item[(ii)] the last root is either simple or the highest root of a symplectic subalgebra, i.e., $\dk(k)=\alpha_j$ or $\dk(k)=\alpha_{\overline{j}}$ for some $1\leq j < n$.
    \item[(iii)] the elements in between satisfy the recursion rule: If $\dk(s)=\alpha_{p,q}$ with $p,q\in \J$, then the next element in the sequence is of the form either $\dk(s+1)=\alpha_{p,q+1}$ or $\dk(s+1)=\alpha_{p+1,q}$; where $x+1$ denotes the smallest element in $\J$ which is bigger than $x$. 
\end{enumerate}
\end{definition}

\begin{example}\label{exa:dyckpathsp6}
For $\sy_6$; the roots can be arranged in form of a triangle as shown below. The Dyck paths are the ones starting at a simple root and ending at one of the edges following the directions indicated by the arrows.\\
\begin{center}
\small\begin{tikzcd}
\alpha_{1,1}\arrow[r] & \alpha_{1,2}\arrow[d] \arrow[r] & \alpha_{1,3}\arrow[d] \arrow[r]& \alpha_{1,\overline{2}}\arrow[d]\arrow[r]& \alpha_{1,\overline{1}}\\
& \alpha_{2,2} \arrow[r] & \arrow[d] \alpha_{2,3}\arrow[r]& \alpha_{2,\overline{2}}\\
                                  & & \alpha_{3,3}
\end{tikzcd}
\end{center}
\end{example}

\begin{definition}\label{def:polytope}
Denote by $\mathbb{D}$ the set of all Dyck paths. For a dominant, integral weight $\lambda = \sum_{i=1}^n m_i \omega_i\in\Lambda^+$, the \emph{symplectic Feigin-Fourier-Littelmann-Vinberg (FFLV) polytope} $\Poly(\lambda)\subset \mathbb{R}^{n^2}_{\geq 0}$ is the polytope of the points $\p=(\p_\alpha)_{\alpha>0}\in\mathbb{R}^{n^2}_{\geq 0}$ such that for all $\dk\in\mathbb{D}$, the following inequalities are satisfied: 
\begin{align}
\begin{cases}
    & \p_{\dk(0)} + \ldots + \p_{\dk(k)} \leq m_i +\ldots + m_j, \quad \text{if} \quad \dk(0)=\alpha_i,\quad   \dk(k)=\alpha_j,\\
    & \p_{\dk(0)} + \ldots + \p_{\dk(k)} \leq m_i +\ldots + m_n,\quad \text{if} \quad \dk(0)=\alpha_i,\quad   \dk(k)=\alpha_{\overline{j}}.
    \end{cases}
\end{align}
\end{definition}
In what follows, we will sometimes also use the shorthand notation $\p_{i,j}=\p_{\alpha_{i,j}}$ and $\p_{i,\overline{j}}=\p_{\alpha_{i,\overline{j}}}$.
\begin{example}
Consider the Dyck paths in Example  \ref{exa:dyckpathsp6}. Here we have $\lambda = m_1\omega_1 +m_2\omega_2 +m_3\omega_3$, so $\Poly(\lambda) \subset \mathbb{R}_{\geq 0}^ 9$ is the polytope defined by all points \[\p=(\p_{1,1},\p_{1,2},\p_{1,3},\p_{1,\overline{2}},\p_{1,\overline{1}},\p_{2,2},\p_{2,3},\p_{2,\overline{2}},\p_{3,3})\in \mathbb{R}_{\geq 0}^ 9\] 
satisfying all inequalities arising from all Dyck paths as seen in Definition \ref{def:polytope} above.
\end{example}

 Let $\set(\lambda)$ be the set of integral points in $\Poly(\lambda)$. For a multi-exponent $\p=(\p_{\alpha})_{\alpha >0}$, $\p_{\alpha}\in \mathbb{Z}_{\geq 0}$, let $f^{\p}$ be the monomial element:
\begin{equation}\label{eqn:monomial}
    f^{\p} = \prod_{\alpha\in \Phi^+}f_{\alpha}^{\p_{\alpha}}\in \Sm(\n^{-}),
\end{equation}
where $\Sm(\n^{-})$ denotes the symmetric algebra of $\n^{-}$.

\begin{theorem}[\cite{FFL11b}]\label{thm:FFLV-degenerate}
The elements $\{f^{\p}\hv^a, \p\in \set(\lambda)\}$ form a basis of $\V_{\lambda}^a$ and $\{f^{\p}\hv, \p\in \set(\lambda)\}$ form a basis of $\V_{\lambda}$ (after fixing a total order for the root vectors in each monomial $f^{\p}$).
\end{theorem}

In what follows, we will refer to the basis $\{f^{\p}\hv, \p\in \set(\lambda)\}$ as the \textit{symplectic FFLV basis}. For any two dominant integral weights $\lambda$ and $\mu$, one has a unique injective homomorphism of modules, $\V_{\lambda+\mu} \hookrightarrow{}\V_{\lambda}\otimes \V_{\mu},\quad \nu_{\lambda + \mu} \mapsto \hv\otimes\nu_{\mu}.$ We end this section by stating an analogous result in the PBW degenerate case: 

\begin{lemma}[\cite{FFL11b}]\label{lem:module-embeddings}
 For any two dominant, integral weights $\lambda$ and $\mu$, there exists an injective homomorphism of modules:
 \begin{align*}
     \V_{\lambda+\mu}^a \hookrightarrow{} \V_{\lambda}^a\otimes \V_{\mu}^a,\quad \nu_{\lambda + \mu}^a \mapsto \hv^a\otimes 
     \nu_{\mu}^a.
\end{align*}
\end{lemma}

\section{The [Symplectic FFLV Basis] ---   [PBW Tableaux] Correspondence}\label{sec:tableaux}
\numberwithin{equation}{section}
 In this section, we define a set of PBW-semistandard tableaux that is in a one-to-one correspondence with the symplectic FFLV basis. We explicitly construct the corresponding maps, first for fundamental weights and then we later generalise to any dominant, integral weight. These tableaux take entries in the set $\N :=\{1,\ldots,n,\overline{n}, \ldots, \overline{1}\}$, with the usual order: $1<\cdots<n<\overline{n}< \cdots< \overline{1}$. Before we proceed, we first recall a few preliminary notions on Young diagrams and tableaux: 

\subsection{Young Diagrams and Tableaux}\label{subsec:diagrams}
Given a partition $\lambda=(\lambda_1\geq\lambda_2\geq \ldots\geq \lambda_n\geq 0)$, the corresponding \emph{Young diagram}, which we denote by $Y_{\lambda}$, is a finite collection of boxes arranged in left-justified rows. The rows and columns in $Y_{\lambda}$ are numbered from top to bottom and left to right respectively. Therefore, to every box of $Y_{\lambda}$ we assign a pair $(i,j)$ for $1\leq i \leq n$ and $1\leq j \leq \lambda_1$. For example, $Y_{(6,4,4,2)}$ is the following diagram:
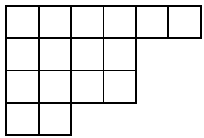
\begin{center}
$\yng(6,4,4,2)$
\end{center}
A \emph{Young tableau} $\T_{\lambda}$ is a filling of $Y_{\lambda}$ with numbers $\T_{i,j}\in \N$, where $\T_{i,j}$ denotes the number put in the box labelled by the pair $(i,j)$. We call the partition $\lambda$ the \emph{shape} of the tableau $\T_{\lambda}$. A tableau $\T_{\lambda}$ is called a \emph{semistandard Young tableau} if the entries $\T_{i,j}$ are such that they are strictly increasing down the columns and weakly increasing across the rows from left to right. The tableaux we define below are analogues of the semistandard Young tableaux.

Finally, to a dominant, integral weight $\lambda =\sum_{k=1}^n m_k \omega_k$, we assign a partition $$(m_1+m_2+\ldots+ m_n\geq  m_2+\ldots+m_n\geq \ldots \geq m_n),$$
which we label by the same symbol $\lambda$. Moreover, from a given Young diagram $Y_{\lambda}$, we can recover the weight $\lambda$ by associating a fundamental weight $\omega_k$ to each column of length $k$, and then summing up. For example, the weight corresponding to the Young diagram $Y_{(6,4,4,2)}$ shown above is $\lambda=2\omega_1+2\omega_3+2\omega_4$.

\subsection{The Case of Fundamental Weights} In this subsection, we set $\lambda$ to be a fundamental weight, i.e., $\lambda=\omega_k$ for $1\leq k\leq n$. To such a weight, we associate as described above, a partition 
\begin{equation}\label{eqn:singlecolumn}
    \lambda =\underbrace{(1,\ldots,1)}_{k- \text{times}}.
\end{equation}
The Young diagram of such a partition is just a single column of length $k$. We have the following definition:

\begin{definition}\label{def:fundamentals}
For a partition $\lambda$ given in \eqref{eqn:singlecolumn} above, the \emph{symplectic PBW tableau} $\T_{\lambda}$ is the filling of the corresponding Young diagram $Y_{\lambda}$ with numbers $\T_i\in \N$ such that:
\begin{enumerate}
    \item[(i)] if $\T_i\leq k$, then $\T_i=i$, (in this case we say that the entry $\T_i$ is \emph{at its position}),
    \item[(ii)] if $i_1<i_2$ and $\T_{i_1}\neq i_1$, then $\T_{i_1} >\T_{i_2}$, and 
    \item[(iii)] if there exist $i,i'$ with  $\T_{i}=i$ and $\T_{i'}=\overline{i}$, then $i' < i$, whenever $i<k$.
\end{enumerate}
\end{definition}

\begin{example}\label{exa:tableaux-fundemental}
For $n=3$ and $\lambda= (1,1,1)$, all the possible symplectic PBW tableaux are: \\
$$
\small\begin{ytableau}
1\\2\\3
\end{ytableau},\,
\begin{ytableau}
1\\2\\\overline{3}
\end{ytableau},\,
\begin{ytableau}
1\\\overline{3}\\3
\end{ytableau},\,
\begin{ytableau}
1\\\overline{2}\\3
\end{ytableau},\,
\begin{ytableau}
1\\\overline{2}\\\overline{3}
\end{ytableau},\,
\begin{ytableau}
\overline{3}\\2\\3
\end{ytableau},\,
\begin{ytableau}
\overline{2}\\2\\3
\end{ytableau},\,
\begin{ytableau}
\overline{2}\\2\\ \overline{3}
\end{ytableau},\\
\begin{ytableau}
\overline{2}\\\overline{3}\\3
\end{ytableau},\,
\begin{ytableau}
\overline{1}\\2\\3
\end{ytableau},\,
\begin{ytableau}
\overline{1}\\2\\\overline{3}
\end{ytableau},\,
\begin{ytableau}
\overline{1}\\\overline{3}\\3
\end{ytableau},\,
\begin{ytableau}
\overline{1}\\\overline{2}\\3
\end{ytableau},\,
\begin{ytableau}
\overline{1}\\\overline{2}\\\overline{3}
\end{ytableau}.$$
\end{example}
\vspace*{2mm}
Let $\textsc{SyP}_{\lambda}$ be the set of all elements $f^{\p} \cdot \hv $ with the product $f^{\p}$ given as in Equation \eqref{eqn:monomial}.
Recall that $\textsc{SyP}_{\lambda}$ is the symplectic FFLV basis for the $\syp$-module $\V_{\lambda}$. Also, let $\textsc{SyT}_{\lambda}$ denote the set of all symplectic PBW tableaux of the shape $\lambda$ that is given in \eqref{eqn:singlecolumn}. We want to establish a weight preserving bijection between $\textsc{SyP}_{\lambda}$ for $\lambda=\omega_k$ and $\textsc{SyT}_{\lambda}$.  To do this, we first describe in the following definition, how to assign a symplectic PBW tableau to each element of $\textsc{SyP}_{\lambda}$. 

\begin{definition}\label{def:assignment1}
Let $\lambda=\omega_k$ be a fundamental weight for $1\leq k\leq n$. Let $t_{\lambda}$ denote the highest weight single column tableau, i.e., the tableau with $u$ appearing in box $u$ for all $1\leq u \leq k$. To an element $f^{\p}\cdot \nu_{\lambda}$, we assign an element $f^{\p}\cdot t_{\lambda}$ by applying each operator $f_{u,v}$ appearing in $f^{\p}$ to entry $u$ of $t_{\lambda}$ according to the following rule:
\begin{align}\label{eqn:assignment}
   f_{u,v} \cdot\small \begin{ytableau}\text{\tiny$u$} \end{ytableau}= \begin{cases}    
                       \begin{ytableau}
                       \text{\tiny $v+1$} \end{ytableau}\quad\, \text{if} \quad k \leq v \leq n,\\
                       \begin{ytableau}
                       \text{\tiny $v$} \end{ytableau} \,\quad \text{if} \quad \overline{n-1} \leq v \leq \overline{1},
                      \end{cases}
\end{align}
where $n+1=\overline{n}$.
\end{definition}
In the following example we illustrate the above rule:

\begin{example}\label{exa:fundamentalassignment}
Consider $\lambda=\omega_3$ and $n=3$. We will describe how the fifth tableau in the list of tableaux in Example \ref{exa:tableaux-fundemental} is obtained by the above rule. For this, consider the element $f_{2,\overline{2}}f_{3,3}\cdot \nu_{\omega_3}$ in the symplectic FFLV basis for the $\sy_6$-module $\V_{\omega_3}$. One has:
\[f_{2,\overline{2}}f_{3,3} \cdot t_{\lambda}= f_{2,\overline{2}}f_{3,3} \cdot \small\begin{ytableau}
1\\2\\3\end{ytableau} =f_{2,\overline{2}} \cdot \small\begin{ytableau}
1\\2\\\overline{3}\end{ytableau}= \begin{ytableau}
1\\\overline{2}\\\overline{3}
\end{ytableau}.\]
\end{example}
Recall the basis $\{\varepsilon_1,\ldots,\varepsilon_n\}$ for $\h^*$. Let $\textsf{wt}(\textsc{z})$ denote the weight of an object $\textsc{z}$ (for example, tableau, highest weight vector, etc). In the following definition, we assign weights to the different objects we are dealing with.

\begin{definition}\label{def:symplecticweight}
Let $\lambda=\omega_k$ be a fundamental weight. Let $\T_{\lambda}$ be a symplectic PBW tableau of shape $\lambda$, $\N^+ := \{i\in \{1,\ldots,n\} \,\,:\,\,i\in \T_{\lambda} \}$ and $\N^-:=\{j\in \{1,\ldots,n\}\,\, :\,\, \overline{j}\in  \T_{\lambda}\}$. The weight of $\T_{\lambda}$ is given by:
\begin{equation*}
    \textsf{wt}(\T_{\lambda}) = \sum_{i\in \N^+}\varepsilon _i - \sum_{j\in \N^-}\varepsilon _j.
\end{equation*}
The weights of the root vectors $f_{i,j}$ and $f_{i,\overline{j}}$ are their usual weights in the Lie algebra $\n^-$, i.e., 
$$
\begin{array}{rcll}
\textsf{wt}(f_{i,j}) &=& \varepsilon _i -\varepsilon_j, & \text{ for } 1\leq i\leq j < n,\\
\textsf{wt}(f_{i,\overline{j}}) & = & -\varepsilon _i -\varepsilon_j,& \text{ for } 1\leq i\leq j \leq n.
\end{array}
$$
\end{definition}

\begin{remark}
From the above definition, it follows that for the product $f^{\p}=\prod_{\alpha > 0}f_{\alpha}^{\p_{\alpha}}$, the weight is:
\begin{equation*}
   \textsf{wt}(f^{\p}) = \sum_{\alpha> 0}\p_{\alpha}\cdot \textsf{wt}(f_{\alpha}).
\end{equation*}  
We also have $\textsf{wt}(f^{\p}\cdot t_{\lambda}) = \textsf{wt}(f^{\p})+\textsf{wt}(t_{\lambda})$ and $\textsf{wt}(t_{\lambda}) =\textsf{wt}(\hv)$. Notice that the weight of the element $f^{\p}\cdot \nu_{\lambda}$ as given here is its actual weight in the module $\V_{\lambda}$.
\end{remark}

\begin{example}
For $\lambda=\omega_3$, consider the fifth tableau $\T_{\lambda}$ in the list given in Example \ref{exa:tableaux-fundemental}. We see that $\textsf{wt}(\T_{\lambda})=\varepsilon_1-\varepsilon_2-\varepsilon_3$. 
For the element $f_{2,\overline{2}}f_{3,3} \cdot \nu_{\lambda}$, we have: $$\textsf{wt}(f_{2,\overline{2}}f_{3,3} \cdot \nu_{\lambda})=-\varepsilon_2-\varepsilon_2-\varepsilon_3-\varepsilon_3+\varepsilon_1+\varepsilon_2+\varepsilon_3=\varepsilon_1-\varepsilon_2-\varepsilon_3.$$ Notice that we have the equality $\textsf{wt}(f_{2,\overline{2}}f_{3,3} \cdot \nu_{\lambda})=\textsf{wt}(\T_{\lambda})$ (compare with Example \ref{exa:fundamentalassignment}).
\end{example}

 We prove the following result:
\begin{proposition}\label{pro:fundamentals} For $\lambda=\omega_k$ a fundamental weight, the set $\textsc{SyP}_{\lambda}$ is in a weight preserving one-to-one correspondence with the set $\textsc{SyT}_{\lambda}$. 
\end{proposition}
\begin{proof}
Define the assignment:
\begin{align*}
    \theta_1: \textsc{SyP}_{\lambda} \longrightarrow \textsc{SyT}_{\lambda},\quad
    f^{\p} \cdot \hv \longmapsto f^{\p} \cdot t_{\lambda},
\end{align*}
where $f^{\p} \cdot t_{\lambda}$ is given according to Definition \ref{def:assignment1}. 

We will show that $\theta_1$ is a well defined map. Since the elements of the set $\textsc{SyP}_{\lambda}$ form a basis for the $\syp$-module $\V_{\lambda}$ for $\lambda=\omega_k$ according to Theorem \ref{thm:FFLV-degenerate}, it suffices to show that for any $f^{\p} \cdot \hv\in \textsc{SyP}_{\lambda}$, we have $\theta_1(f^{\p} \cdot \hv) = f^{\p} \cdot t_{\lambda}\in \textsc{SyT}_{\lambda}$. Let $f^{\p}\cdot \nu_{\lambda} = f_{i_1,j_1}\cdots f_{i_s,j_s} \cdot \nu_{\lambda}\in \textsc{SyP}_{\lambda}$. Let $\alpha_{i_1,j_1}, \ldots, \alpha_{i_s,j_s}$ be the roots corresponding to the root vectors appearing in $f^{\p}\cdot \nu_{\lambda}$. Now since $\lambda=\omega_k$ is a fundamental weight, then according to Definition 
\ref{def:polytope}, all inequalities are of the following form:
\[\ldots+\p_{i_l,j_l}+\ldots\leq 1,\]
for $1\leq l\leq s$. This implies that the roots $\alpha_{i_1,j_1}, \ldots, \alpha_{i_s,j_s}$ are not pairwise on a common Dyck path. This in turn means that $i_1\neq \ldots\neq i_s$, since at least two roots would lie on the same Dyck path otherwise. Now pairs of roots that don't lie on the same Dyck path are of the form:
\begin{enumerate}
    \item[(i)] $\alpha_{p,q}$ and $\alpha_{p+1,q-1}$ which implies $p<p+1$ and $q>q-1$, or
    \item[(ii)] $\alpha_{p,q}$ and  $\alpha_{p-1,q+1}$ which implies $p-1<p$ and $q+1>q$,
\end{enumerate}
 where $x+1$ is an element in $\J=\{1,\ldots,n,\overline{n-1},\ldots,\overline{1}\}$ that is bigger than $x$ while $x-1$ is an element in $\J$ that is smaller than $x$. Since this is true for all pairs of roots and we have $1\leq i_1\neq \ldots\neq i_s \leq k$, reordering these indices to have $1\leq i_1<\cdots< i_s\leq k$ implies $\overline{1}\geq j_1 >\cdots > j_s\geq k$. 

Now since we have $i_1\neq \cdots \neq i_s$, the operators $f_{i_1,j_1},\ldots,f_{i_s,j_s}$ each act on a different entry of the single column highest weight tableau $t_{\lambda}$ once. Let
$$f_{i_1,j_1} \cdot \tiny\begin{ytableau}i_1 \end{ytableau} = \tiny\begin{ytableau}i_1' \end{ytableau}, \ldots ,f_{i_s,j_s} \cdot \tiny\begin{ytableau}i_s \end{ytableau}=\tiny\begin{ytableau}i_s' \end{ytableau},$$ then we have $i_1'>\cdots >i_s'$ according to \eqref{eqn:assignment}. We note that the entries of $t_{\lambda}$ which are not acted upon remain at their positions, so condition (i) of Definition \ref{def:fundamentals} is satisfied. The elements $i_1',\ldots ,i_s'$ are the ones that are not at their positions. Condition (ii) is therefore satisfied since we have $i_1'>\cdots >i_s'$. We are left with showing that condition (iii) holds true. For an entry $m$ in the tableau $f^{\p} \cdot t_{\lambda}$ with $m< k$, we need to check that if $\overline{m}$ exists in the tableau, then its position is above that of $m$. Consider $f_{i_1,j_1}\cdots f_{i_s,j_s} \cdot \tiny\begin{ytableau} m \end{ytableau}.$ Assume there exists $j_p\in \{j_1,\ldots,j_s\}$ such that $j_p=\overline{m}$. 
If $m\in \{i_1,\ldots,i_s\}$, then we have $f_{m,\overline{m}} \cdot \tiny \begin{ytableau} m  \end{ytableau}= \tiny\begin{ytableau} \overline{m}  \end{ytableau}$. Hence $m$ will not appear in the resulting tableau. In case $m\notin  \{i_1,\ldots,i_s\}$, then we have $f_{i_p,\overline{m}}\cdot \tiny\begin{ytableau} i_p \end{ytableau}=\tiny\begin{ytableau} \overline{m}  \end{ytableau}$. But $i_p<m$ from \eqref{eqn:roots}, so $\overline{m}$ is above $m$, and we are done.

We also define the assignment:
\begin{align*}
    \theta_2: \textsc{SyT}_{\lambda} \longrightarrow \textsc{SyP}_{\lambda},\quad
    \T_{\lambda} \longmapsto f^{\p}\cdot \hv =f_{i_1,j_1}\cdots f_{i_s,j_s} \cdot \hv,
\end{align*}
by associating a root vector $f_{i_l,j_l}$ to an element of $\T_{\lambda}$ that is not at its position for all $l$ with $1\leq l\leq s$. Let $x_1,\ldots,x_s$ denote the entries of $\T_{\lambda}$ that are not at their positions with $x_1>\cdots>x_s$. Let $i_1,\ldots,i_s$ denote the box numbers of the entries $x_1,\ldots,x_s$ respectively. Then the operator $f_{i_l,j_l}$ for $1\leq l\leq s$ is obtained by the following rule:
\begin{align}\label{eqn:2}
 f_{i_l,j_l}  = \begin{cases}
             f_{i_l,x_l-1}\quad \, \text{if} \,\,\quad \overline{n} \geq x_l > k,\\
             f_{i_l,x_l}\quad\quad\,\,  \text{if} \,\,\quad \overline{1} \geq x_l \geq\overline{n-1}.          \end{cases}
\end{align}
We claim that $\theta_2$ is a well defined map. For this, we will show that $\theta_2(\T_{\lambda})=f_{i_1,j_1}\cdots f_{i_s,j_s} \cdot \hv\in \textsc{SyP}_{\lambda}$.
We have $\overline{1}\geq j_1>\cdots>j_s>k$, since the elements $x_1,\ldots,x_s$ are not at their positions in $\T_{\lambda}$. We also have $1\leq i_1<\cdots <i_s\leq k$. Therefore, for $1\leq l\leq s$, each root $\alpha_{i_l,j_l}$ corresponding to the root vector $f_{i_l,j_l}$ lies in some Dyck path with no two distinct roots lying in a common Dyck path. So, the point $(\ldots,\p_{i_l,j_l},\ldots)$ with $\p_{i_l,j_l}=1$ satisfies an inequality of the form:
\[\cdots + \p_{i_l,j_l} + \cdots \leq 1,\]
therefore $f_{i_1,j_1}\cdots f_{i_s,j_s} \cdot \hv\in \textsc{SyP}_{\lambda}$. The claim is proved. 

Now we will check that $\theta_1\circ \theta_2 = \theta_2 \circ \theta_1 = \text{id}$, where $\text{id}$ denotes the identity map. Consider $\theta_1\circ \theta_2(\T_{\lambda})=
\theta_1(f_{i_1,j_1}\ldots f_{i_s,j_s}\cdot \nu_{\lambda})$ with $f_{i_l,j_l}$ for $1\leq l\leq s$ obtained as in \eqref{eqn:2} above. Then from \eqref{eqn:assignment}, we have:
\begin{align*}
   f_{i_l,j_l} \cdot {\tiny\begin{ytableau} i_l \end{ytableau}}= \begin{cases}    
                        {\tiny\begin{ytableau} x_l \end{ytableau}} \qquad\, \text{if} \quad k \leq j_l \leq n,\\
                         {\tiny\begin{ytableau} x_l \end{ytableau}}\quad\quad\,\, \text{if} \quad \overline{n-1} \leq j_l \leq \overline{1}.
                      \end{cases}
        \end{align*}
Therefore, we have $\theta_1(f_{i_1,j_1}\ldots f_{i_s,j_s} \cdot \hv)=\T_{\lambda} \Rightarrow \theta_1\circ \theta_2 =\text{id}$. Now consider $\theta_2\circ \theta_1(f_{i_1,j_1}\ldots f_{i_s,j_s} \cdot \hv)=\theta_2(\T_{\lambda}) $ with the entries $x_1,\ldots,x_s$ of $\T_{\lambda}$ obtained from $f_{i_l,j_l}$ for $1\leq l\leq s$ according to \eqref{eqn:assignment}. Now applying $\theta_2$ to $\T_{\lambda}$, we get: 
\begin{align*}
 f_{i_l,j_l}  = \begin{cases}
             f_{i_l,x_l+1-1}\quad \, \text{if} \,\,\quad \overline{n} \geq x_l > k,\\
             f_{i_l,x_l}\quad\quad\quad\,  \text{if} \,\,\quad \overline{1} \geq x_l \geq\overline{n-1},            \end{cases}
\end{align*}
therefore we have $\theta_2(\T_{\lambda})=f_{i_1,j_1}\ldots f_{i_s,j_s} \cdot \hv\Rightarrow \theta_2 \circ \theta_1 = \text{id}$. We have therefore shown that $\theta_1\circ \theta_2=\theta_2\circ \theta_1=\text{id}$, which means that the maps $\theta_1$ and $\theta_2$ are inverse to each other. So, we have constructed the required bijection.

We are now left with proving that the defined maps are weight preserving.  For this, it suffices to show that the map $\theta_1$
is weight preserving, i.e., that  $\textsf{wt}(\theta_1(f^{\p} \cdot \hv))=\textsf{wt}(f^{\p} \cdot \hv)$. 
Indeed we have:
$$
\textsf{wt}(\theta_1(f^{\p} \cdot \hv)) = \textsf{wt}(f^{\p} \cdot t_{\lambda})
    = \textsf{wt}(f^{\p}) + \textsf{wt}(t_{\lambda})
    =\textsf{wt}(f^{\p}) + \textsf{wt}(\hv)
    =\textsf{wt}(f^{\p} \cdot \hv).$$ 
\end{proof}

\subsection{The Case of Dominant Integral Weights}
In this subsection, we extend results from the previous subsection on the case of fundamental weights to the case of any dominant integral weight $\lambda=\sum_{k=1}^nm_k\omega_k$ of $\syp$. As before, we will denote by $\lambda$ the corresponding partition, i.e., $\lambda = (\lambda_1\geq \lambda_2\geq \ldots \geq \lambda_n \geq 0)$, where $\lambda_i=m_i+\ldots+m_n$ for $1\leq i\leq n$.

\begin{definition}\label{eqn:symsspbw}
For a dominant integral weight $\lambda$, a \emph{symplectic PBW tableau} $\T_{\lambda}$ whose shape is the corresponding partition $\lambda$, is a filling of the corresponding Young diagram $Y_{\lambda}$ with numbers $\T_{i,j}\in \N$ such that for $\mu_j$, the length of the $j$-th column, we have:
 \begin{enumerate}
  \item[(i)] if $\T_{i,j}\leq \mu_j$, then $\T_{i,j}=i$,
  \item[(ii)] if $\T_{i_1,j}\neq i_1$, and $i_2>i_1$, then $\T_{i_1,j}>\T_{i_2,j}$, 
  \item[(iii)] if $\T_{i,j}= i$, and $\exists \,\, i'$ such that $\T_{i',j}=\overline{i}$, then $i'<i$.
 \end{enumerate}
 
We say that a symplectic PBW tableau $\T_{\lambda}$ is \emph{PBW-semistandard} if in addition, the following condition is satisfied: 

\begin{enumerate}
 \item[(iv)] for every $j>1$ and every $i$, $\exists \,\, i'\geq i$ such that $\T_{i',j-1}\geq \T_{i,j}$.
\end{enumerate}
\end{definition}

\begin{example}\label{exa:tableaux}
For $n=2$, and $\lambda =(2,1)$ (i.e., $\lambda = \omega_1 + \omega_2$), the set of all the 16 symplectic PBW-semistandard tableaux is the one given below:
$$
\small\begin{ytableau}
1 & 1\\
2
\end{ytableau},
\begin{ytableau}
1 & 2\\
2
\end{ytableau},
\begin{ytableau}
1 & 1\\
\overline{2}
\end{ytableau},
\begin{ytableau}
1 & 2\\
\overline{2}
\end{ytableau},
\begin{ytableau}
1 & \overline{2}\\
\overline{2}
\end{ytableau},
\begin{ytableau}
\overline{2} & 1\\
2
\end{ytableau},
\begin{ytableau}
\overline{2} & 2\\
2
\end{ytableau},
\begin{ytableau}
\overline{2} & \overline{2}\\
2
\end{ytableau},
$$
$$
\small\begin{ytableau}
\overline{1} & 1\\
2
\end{ytableau},
\begin{ytableau}
\overline{1} & 2\\
2
\end{ytableau},
\begin{ytableau}
\overline{1} & \overline{2}\\
2
\end{ytableau},
\begin{ytableau}
\overline{1} & \overline{1}\\
2
\end{ytableau},
\begin{ytableau}
\overline{1} & 1\\
\overline{2}
\end{ytableau},
\begin{ytableau}
\overline{1} & 2\\
\overline{2}
\end{ytableau},
\begin{ytableau}
\overline{1} & \overline{2}\\
\overline{2}
\end{ytableau},
\begin{ytableau}
\overline{1} & \overline{1}\\
\overline{2}
\end{ytableau}.
$$

The following symplectic PBW tableaux are not PBW-semistandard:
\begin{center}
   $\small\begin{ytableau}
1 & \overline{2}\\
2
\end{ytableau}$ , $\small\begin{ytableau}
1 & \overline{1}\\
2
\end{ytableau}$, $\small\begin{ytableau}
1 & \overline{1}\\
\overline{2}
\end{ytableau}$, $\small\begin{ytableau}
\overline{2} & \overline{1}\\
2
\end{ytableau}$.
\end{center}
\end{example}

Denote by $\SyST_{\lambda}$ the set of all symplectic PBW-semistandard tableaux of shape $\lambda$ on the set $\N$ as above.

We introduce a total order on the operators $f_{i,j},f_{i,\overline{j}}$ as follows:\\
We say $f_{i_1, j_1}>f_{i_2,  j_2}$ if either $i_1< i_2$ 
or $i_1 = i_2$ and $j_1 < j_2$. For example, we have $f_{1\overline{1}}>f_{22}$ and $f_{12}>f_{1\overline{1}}$. We now order our operators in the product
$f^{\p}=\prod_{\alpha > 0}f_{\alpha}^{\p_{\alpha}}$ according to this total order.\\

In the following definition, we extend the assignment described in Definition \ref{def:assignment1} to the case of dominant integral weights.
\begin{definition}\label{def:assignment2}
Let $\lambda$ be a dominant integral weight and $t_{\lambda}$ be a highest weight tableau, i.e., a tableau with one's in the first row, two's in the second row, and so on. We define the assignment $f^{\p}\cdot t_{\lambda}$ as follows. Apply the operators in the ordered product $f^{\p}$ starting with the smallest one. An operator $f_{i,j}$ acts on entry $i$ in column $c$ whenever $j\geq \mu_c$, where $c$ is the first column from the left where this is true. The assignment $f^{\p} \cdot t_{\lambda}$ then narrows down to the assignment $f_{i,j}\cdot \tiny\begin{ytableau} i
\end{ytableau}$ of each operator $f_{i,j}$ in the product $f^{\p}$ only once on the entry $i$ in the best choice column $c$ of $t_{\lambda}$ according to rule \eqref{eqn:assignment} in Definition \ref{def:assignment1}.
\end{definition}

\begin{example}\label{exa:assignment}
For $\sy_4$ and $\lambda =\omega_1 +\omega_2$, one has 16 integral points of the polytope $\Poly(\lambda)$. This leads to the following set of monomials: 
\begin{align*}
    &\{1,f_{11}, f_{22}, f_{11}f_{22}, f_{12}f_{22}, f_{12},f_{11}f_{12}, f_{12}^2, f_{1\overline{1}}, f_{11}f_{1\overline{1}},f_{12}f_{1\overline{1}},f_{1\overline{1}}^2,f_{1\overline{1}}f_{2\overline{2}},f_{11}f_{1\overline{1}}f_{22},\\
    &f_{12}f_{1\overline{1}}f_{22}, f_{1\overline{1}}^2f_{22}\},
\end{align*}
each of them corresponding to the symplectic PBW-semistandard tableau appearing in the same position in the list of tableaux given in Example \ref{exa:tableaux}. For an illustration of how the assignment described in Definition \ref{def:assignment2} works, consider the second last monomial in the list above. Then one has:
$$f_{12}f_{1\overline{1}}f_{22}\cdot \small\begin{ytableau}
1 & 1\\
2
\end{ytableau}=f_{12}f_{1\overline{1}}\cdot\small\begin{ytableau}
1 & 1\\
\overline{2}
\end{ytableau}=
f_{12}\cdot\small\begin{ytableau}
\overline{1} & 1\\
\overline{2}
\end{ytableau}=
\small\begin{ytableau}
\overline{1} & \overline{2}\\
\overline{2}
\end{ytableau}.$$
The resulting tableau is the second last one in the list of tableaux in Example \ref{exa:tableaux}.
\end{example}

\begin{proposition}\label{pro:to_tableaux}
The assignment
\begin{align*}
 \phi: \textsc{SyP}_{\lambda} \longrightarrow \SyST_{\lambda},\quad
 f^{\p} \cdot \hv &\longmapsto f^{\p} \cdot t_{\lambda},
\end{align*}
where $f^{\p} \cdot t_{\lambda}$ is given according to Definition \ref{def:assignment2} is a well defined map. 
\begin{proof}
Since the elements of the set $\textsc{SyP}_{\lambda}$ are a basis for the $\syp$-module $\V_{\lambda}$ according to Theorem \ref{thm:FFLV-degenerate}, it suffices to show that for an arbitrary element $f^{\p} \cdot \hv\in \textsc{SyP}_{\lambda}$, we have $\phi(f^{\p} \cdot \hv)=f^{\p} \cdot t_{\lambda}\in\SyST_{\lambda}$. Let $f^{\p}=f_{i_1,j_1}\cdots f_{i_s,j_s}$ be the ordered product, with $f_{i_1,j_1}\geq \cdots \geq f_{i_s,j_s}$. We begin `acting' with the smallest operator $f_{i_s,j_s}$ in the first column $c_1$ from the left for which $j_s\geq \mu_{c_1}$. We then proceed to the next smallest one $f_{i_{s-1},j_{s-1}}$. If $i_{s-1}<i_s$ and $j_{s-1}>j_s$, then $f_{i_{s-1},j_{s-1}}$ also acts in the same column. Let $f_{i_{s-k},j_{s-k}}\cdots f_{i_s,j_s}$ be the product of the operators which act in the same column. The result of this product satisfies all conditions of symplectic PBW tableaux defined on columns according to Proposition  \ref{pro:fundamentals}. Now let $f_{i_{s-k-1},j_{s-k-1}}$ be the next smallest operator for which $i_{s-k-1}\leq i_{s-k}$ and $j_{s-k-1}\leq j_{s-k}$. This operator acts in the next column $c_2$ to the right of the column $c_1$. Because of the above argument, it suffices to check that the two columns $c_1$ and $c_2$ satisfy condition (iv) of Definition \ref{eqn:symsspbw}. 

If $\mu_{c_1}\leq j_{s-k}\leq n$, we have $\mu_{c_2} \leq j_{s-k-1}\leq j_{s-k}\leq n$. So under the assignment $\phi$, we have:
$$f_{i_{s-k},j_{s-k}} \cdot \hv \longmapsto j_{s-k}+1 \quad \text{and}\quad   f_{i_{s-k-1},j_{s-k-1}} \cdot \hv \longmapsto j_{s-k-1}+1.$$
We have $i_{s-k-1}\leq i_{s-k}$ and $j_{s-k-1}\leq j_{s-k} \Rightarrow j_{s-k-1}+1\leq j_{s-k}+1$. We thus have the entry $j_{s-k}+1$ in row $i_{s-k}$ and column $c_1$ and the entry $j_{s-k-1}$ in row $i_{s-k-1}$ and column $c_2$, such that $j_{s-k-1}+1\leq j_{s-k}+1$, which implies that condition (iv) of Definition \ref{eqn:symsspbw} is satisfied in this case. In like manner, we check the other cases as follows:

If $\overline{n-1}\leq j_{s-k}\leq \overline{1}$, then also $\mu_{c_2} \leq j_{s-k-1}\leq j_{s-k}\leq \overline{1}$. Here we have two cases:\\
(i) if $\mu_{c_2} \leq j_{s-k-1}\leq n$, then we have:
$$f_{i_{s-k},j_{s-k}} \cdot \hv \longmapsto j_{s-k}  \quad \text{and}\quad f_{i_{s-k-1},j_{s-k-1}} \cdot \hv \longmapsto j_{s-k-1}+1.$$
So, we have $i_{s-k-1}\leq i_{s-k}$ and $j_{s-k-1} < j_{s-k} \Rightarrow j_{s-k-1}+1\leq j_{s-k}$. Hence condition (iv) of Definition \ref{eqn:symsspbw} is again satisfied following the above argument.\\
(ii) if $\overline{n-1}\leq j_{s-k-1}\leq \overline{1}$, then we have:
 $$f_{i_{s-k},j_{s-k}} \cdot \hv\longmapsto j_{s-k} \quad \text{and}\quad  f_{i_{s-k-1},j_{s-k-1}} \cdot \hv\longmapsto j_{s-k-1}.$$
Hence we again have $i_{s-k-1}\leq i_{s-k}$ and $j_{s-k-1} \leq j_{s-k}$. Therefore condition (iv) of Definition \ref{eqn:symsspbw} still holds true.
We continue in the same way until all the operators are applied. 
\end{proof}
\end{proposition}

We now construct the inverse map to the map $\phi$ described in Proposition \ref{pro:to_tableaux}. To do this, we describe in the following definition how to recover an element $f^{\p} \cdot \hv$ from a tableaux $\T_{\lambda}\in \SyST_{\lambda}$. 

\begin{definition}\label{def:to-monomial}
Given a tableaux $\T_{\lambda}\in \SyST_{\lambda}$, let $h$ denote the entry in the box labelled by the pair $(r,c)$, where $r$ is the row number and $c$ is the column number. Let $\mu_c$ denote the length of the column $c$. An element $f^{\p}\cdot \hv$ is obtained from $\T_{\lambda}$ in the following way.

\begin{enumerate}
    \item[(a)] To each entry $h$ in $\T_{\lambda}$ that is greater than $\mu_c$, apply the following assignment:
\begin{align}\label{eqn:to-monomial}
\begin{cases}
             h\mapsto f_{r,h-1}\,\,\,\,\,\, \text{if} \,\,\quad\mu_c< h\leq \overline{n} ,\\
             h\mapsto f_{r,h}\quad\quad  \text{if} \,\,\quad \overline{n-1}\leq h\leq\overline{1},
             \end{cases}
\end{align}
where $\overline{n}-1=n$.
\item[(b)] Let $h'$ denote $h$ or $h-1$ as the case may be after applying the above assignment. We put together the operators $f_{r,h-1}$ and $f_{r,h}$ to obtain the product $f^{\p}=\prod_{h'} f_{r,h'}^{\p_{h'}}$, where $\p_{h'}$ indicates the multiplicity of $f_{r,h'}$. 
\item[(c)] The weight $\lambda$ is obtained from the shape of the tableau $\T_{\lambda}$ as described in Subsection \ref{subsec:diagrams}. This and step (b) above yield the element $f^{\p}\cdot \hv$.
\end{enumerate}
\end{definition}

We prove:
\begin{proposition}\label{pro:to_monomials}
The assignment
\begin{align*}
 \pi: \SyST_{\lambda} \longrightarrow \textsc{SyP}_{\lambda},\quad
 \T_{\lambda} \longmapsto f^{\p} \cdot \hv,
 \end{align*}
where $f^{\p} \cdot \hv$ is obtained according to Definition \ref{def:to-monomial} is a well defined map. 
\begin{proof}
For any $\T_{\lambda}\in \SyST_{\lambda}$, we claim that $\pi(\T_{\lambda})=f^{\p} \cdot \hv\in \textsc{SyP}_{\lambda}$. To prove this, we proceed as follows. Consider any two arbitrary neighbouring columns $j_1$  and $j_2$ in $\T_{\lambda}$. Let $\mu_{j_1}=l$ and
$\mu_{j_2}=s$ such that $1\leq s \leq l \leq n$. Let $\{x_1,\ldots, x_l\}$ be elements from $j_1$ and $\{y_1,\ldots, y_s\}$ be elements from $j_2$. It suffices to consider only those elements that are not at their positions. Let $\{x_{t_1},\ldots, x_{t_k}\}$ be elements from $j_1$ that are not at their positions and likewise $\{y_{r_1},\ldots, y_{r_k}\}$ be elements from $j_2$ that are not at their positions with $1\leq {t_1} <\cdots < {t_k} \leq l$ and $1\leq {r_1} <\cdots < {r_k} \leq s$. According to the definition of a symplectic PBW-semistandard tableau, we have that $\{x_{t_1}>\cdots > x_{t_k}\}$ and $\{y_{r_1}>\cdots > y_{r_k}\}$.

We put the elements in $j_1$ and $j_2$ together and arrange them in descending order. Let $\{x_{t_1},\ldots, x_{t_{z-1}}\}$ be the first $z-1$ elements that lie in column $j_1$. Let $f_{t_1,x_{t_1}'}\cdots f_{t_{z-1},x_{t_{z-1}}'}$ be the corresponding monomial got by applying the map $\theta_2$ from Proposition \ref{pro:fundamentals}.  Now assume the next biggest element $y_{r_z}$ lies in $j_2$. Then there must exist $x_{t_{z+1}}$ with $t_{z+1} \geq r_z$ such that $x_{t_{z+1}} \geq  y_{r_z}$. 

If $l<x_{t_{z+1}}\leq \overline{n}$, then $s<y_{r_z}\leq x_{t_{z+1}}\leq \overline{n}$, so we have $f_{i_1,j_1}=f_{r_z,y_{r_z}-1}$ and $f_{i_2,j_2}=f_{t_{z+1},x_{t_{z+1}}-1}$ according to Equation $\eqref{eqn:to-monomial}$. And the corresponding monomial is: $f_{i_1,j_1}f_{i_2,j_2}=f_{r_z,y_{r_z}-1}f_{t_{z+1},x_{t_{z+1}}-1}$. The roots $\alpha_{r_z,y_{r_z}-1} $ and $\alpha_{t_{z+1},x_{t_{z+1}}-1}$ lie on a common symplectic Dyck path since $t_{z+1}\geq r_z$ and $x_{t_{z+1}}\geq y_{r_z} \Rightarrow x_{t_{z+1}}-1\geq y_{r_z}-1$. It follows that the corresponding point
\[\p= (0,\ldots,0, \p_{r_z,y_{r_z}-1}, 0, \ldots,0, \p_{t_{z+1},x_{t_{z+1}}-1}, 0,\ldots,0)\] with $\p_{r_z,y_{r_z}-1}=1$ and
$\p_{t_{z+1},x_{t_{z+1}}-1}=1$ satisfies the inequality:
$$\ldots + \p_{r_z,y_{r_z}-1} + \ldots + \p_{t_{z+1},x_{t_{z+1}}-1} + \ldots \leq 2.$$ 
If $\overline{n-1}\leq x_{t_{z+1}}\leq \overline{1}$, then $s<y_{r_z}\leq x_{t_{z+1}}\leq \overline{1}$. 
We have two cases:\\
(i) if $s<y_{r_z}\leq \overline{n}$, then $f_{i_1,j_1}f_{i_2,j_2}=f_{r_z,y_{r_z}-1}f_{t_{z+1},x_{t_{z+1}}}$ and the corresponding roots $\alpha_{r_z,y_{r_z}-1}$ and $\alpha_{t_{z+1},x_{t_{z+1}}}$ lie on a common symplectic Dyck path since $t_{z+1}\geq r_z$ and $x_{t_{z+1}}\geq y_{r_z} \Rightarrow x_{t_{z+1}}> y_{r_z}-1$. Also the corresponding point 
\[\p= (0,\ldots,0, \p_{r_z,y_{r_z}-1}, 0, \ldots,0, \p_{t_{z+1},x_{t_{z+1}}}, 0,\ldots,0)\] with $\p_{r_z,y_{r_z}-1}=1$ and $\p_{t_{z+1},x_{t_{z+1}}}=1$ satisfies the inequality:
$$\ldots + \p_{r_z,y_{r_z}-1} + \ldots + \p_{t_{z+1},x_{t_{z+1}}} + \ldots \leq 2.$$
(ii) if $\overline{n-1} \leq y_{r_z}\leq \overline{1}$, then $f_{i_1,j_1}f_{i_2,j_2}=f_{r_z,y_{r_z}}f_{t_{z+1},x_{t_{z+1}}}$ and the corresponding roots $\alpha_{r_z,y_{r_z}}$ and $\alpha_{t_{z+1},x_{t_{z+1}}}$ lie on a common symplectic Dyck path since $t_{z+1}\geq r_z$ and $x_{t_{z+1}}\geq y_{r_z}$. Also the corresponding point \[\p= (0,\ldots,0, \p_{r_z,y_{r_z}}, 0, \ldots,0, \p_{t_{z+1},x_{t_{z+1}}}, 0,\ldots,0)\] with $\p_{r_z,y_{r_z}}=1$ and $\p_{t_{z+1},x_{t_{z+1}}}=1$ satisfies the inequality:
$$\ldots + \p_{r_z,y_{r_z}} + \ldots + \p_{t_{z+1},x_{t_{z+1}}} + \ldots \leq 2.$$

We continue in the same way until all the elements in the columns $j_1$ and $j_2$ are done. We now put together all products of the form $f_{t_1,x_{t_1}'}\cdots f_{t_{z-1},x_{t_{z-1}}'}$ and all products of the form $f_{i_1,j_1}f_{i_2,j_2}$ to obtain the monomial $f^{\p}$ according to Definition \ref{def:to-monomial}. From the shape $\lambda$ of the tableau $\T_{\lambda}$, we recover the highest weight vector $\hv$ as described in Subsection \ref{subsec:diagrams}. Now from the above argument and from Proposition \ref{pro:fundamentals}, it follows that the element $f^{\p}\cdot \hv$ lies in $\textsc{SyP}_{\lambda}$. 
\end{proof}
\end{proposition}

We extend the definition of weights from Definition \ref{def:symplecticweight} to the case of dominant integral weights by considering all columns in the tableaux $\T_{\lambda}$ and $t_{\lambda}$, and all operators appearing in the corresponding elements $f^{\p}\cdot\hv$. We prove:
\begin{theorem}\label{thm:correspondence}
For $\lambda =\sum_{k=1}^{n}m_k\omega_k$ a dominant integral weight, the symplectic FFLV basis is in a weight preserving one-to-one correspondence with the set $\SyST_{\lambda}$ of symplectic PBW-semistandard tableaux of shape $\lambda$ with entries in $\N$. 
\end{theorem}
\begin{proof}
For the one-to-one correspondence, it suffices to prove that for the maps $\phi$ and $\pi$ constructed in Propositions  \ref{pro:to_tableaux} and \ref{pro:to_monomials} respectively, we have $\phi \circ \pi = \pi \circ \phi =\text{id}$, where id is the identity map.\\

Let us begin with proving that $\phi\circ \pi=\text{id}$. We again consider two neighbouring columns $j_1$ and $j_2$ with $\mu_{j_1}\geq \mu_{j_2}$. It suffices to consider elements that are not at their positions as before. As before, let $\{x_{t_1},\ldots,x_{t_{z-1}}\}$ be elements in the column $j_1$, and $y_{r_z}$ the next element which is in the column to the right, $j_2$, such that $\exists \,\, x_{t_{z+1}}$ with $x_{t_{z+1}}\geq y_{r_z}$ and $t_{z+1}\geq r_z$. If $\mu_{j_1}< x_{t_{z+1}} \leq \overline{n}$, then we have 
$\phi\circ \pi (\T_{\lambda}) = \phi (f_{r_z,y_{r_z}-1}f_{t_{z+1},x_{t_{z+1}}-1} \cdot \hv)$, so we have:
$$f_{t_{z+1},x_{t_{z+1}}-1} \cdot \hv \longmapsto x_{t_{z+1}}-1 +1 = x_{t_{z+1}} \,\, \text{and}\,\,
    f_{r_z,y_{r_z}-1} \cdot \hv \longmapsto y_{r_z}-1 +1 = y_{r_z}.$$
Moreover, we have that $f_{r_z,y_{r_z}-1}\geq f_{t_{z+1},x_{t_{z+1}}-1}$ under our total order with equality if and only if $r_z= t_{z+1}$ and $y_{r_z}-1 = x_{t_{z+1}}-1$. Therefore the operator $f_{t_{z+1},x_{t_{z+1}}-1}$ acts only in the left-hand column $j_1$, since $x_{t_{z+1}}-1\geq \mu_{j_1}$ and the operator $f_{r_z,y_{r_z}-1}$ acts in $j_2$ since $y_{r_z}-1\geq \mu_{j_2}$. So, we have $\phi\circ \pi (\T_{\lambda}) = \T_{\lambda}$. If instead $\overline{n-1}< x_{t_{z+1}} \leq \overline{1}$, then $ y_{r_z} \leq x_{t_{z+1}} \leq \overline{1}$, so we have two cases:\\
(i) if $\mu_{j_2} < y_{r_z}\leq \overline{n}$, then 
$\phi\circ \pi (\T_{\lambda}) = \phi (f_{r_z,y_{r_z}-1}f_{t_{z+1},x_{t_{z+1}}})$, so we have:
$$ f_{t_{z+1},x_{t_{z+1}}} \cdot \hv \longmapsto x_{t_{z+1}} \quad \text{and}\quad
    f_{r_z,y_{r_z}-1} \cdot \hv \longmapsto y_{r_z}-1 +1 = y_{r_z}.$$
Again we have that $f_{r_z,y_{r_z}-1} >  f_{t_{z+1},x_{t_{z+1}}-1}$ under our total order. Therefore the operator $f_{t_{z+1},x_{t_{z+1}}}$ acts only in the left-hand column $j_1$, since $x_{t_{z+1}}\geq \mu_{j_1}$ and the operator $f_{r_z,y_{r_z}-1}$ acts in $j_2$ since $y_{r_z}-1\geq \mu_{j_2}$. So again $\phi\circ \pi (\T_{\lambda}) = \T_{\lambda}$.\\
(ii) if $\overline{n-1} < y_{r_z}\leq \overline{1}$, then 
$\phi\circ \pi (\T_{\lambda}) = \phi (f_{r_z,y_{r_z}}f_{t_{z+1},x_{t_{z+1}}})$, so we have:
$$ f_{t_{z+1},x_{t_{z+1}}} \cdot \hv \longmapsto x_{t_{z+1}}\quad \text{and}\quad
    f_{r_z,y_{r_z}} \cdot \hv \longmapsto y_{r_z}.$$
Again we have that $f_{r_z,y_{r_z}} >  f_{t_{z+1},x_{t_{z+1}}}$ under our total order. Therefore the operator $f_{t_{z+1},x_{t_{z+1}}}$ acts only in the left-hand column $j_1$, since $x_{t_{z+1}}\geq \mu_{j_1}$ and the operator $f_{r_z,y_{r_z}}$ acts in $j_2$ since $y_{r_z}\geq \mu_{j_2}$. So again $\phi\circ \pi (\T_{\lambda}) = \T_{\lambda}$.\\

Now let us prove that $\pi \circ \phi$ = \text{id}. Let $f^{\p}=f_{i_1,j_1}\cdots f_{i_s,j_s}$ be the ordered product. Assume  $f_{i_{s-k},j_{s-k}}\cdots f_{i_s,j_s}$ is the product of the operators which act in the same column $j_1$. Let  $f_{i_{s-k-1},j_{s-k-1}}$ be the smallest operator for which $i_{s-k-1}\leq i_{s-k}$ and $j_{s-k-1}\leq j_{s-k}$.  This operator acts in the right-hand column $j_2$. If $\mu_{j_1}\leq j_{s-k}\leq n$, then also $\mu_{j_2} < j_{s-k-1}\leq j_{s-k}\leq \overline{n}$. So, we have:
\begin{align*}
\pi \circ \phi (f_{i_{s-k-1},j_{s-k-1}}f_{i_{s-k},j_{s-k}} \cdot \hv)&=\pi ((i_{s-k-1},j_{s-k-1}+1),(i_{s-k},j_{s-k}+1)),\\
 &=f_{i_{s-k-1},j_{s-k-1}}f_{i_{s-k},j_{s-k}} \cdot \hv,    
\end{align*}
where here the pair $(i,j)$ means that at position $i$ of a respective column, we have entry $j$. 
If $\overline{n-1}\leq j_{s-k}\leq \overline{1}$, then also $\mu_{j_2} < j_{s-k-1}\leq j_{s-k}\leq \overline{1}$. So, we have two cases:\\
(i) if $\mu_{j_2} < j_{s-k-1}\leq n$, then:
\begin{align*}
    \pi \circ \phi (f_{i_{s-k-1},j_{s-k-1}}f_{i_{s-k},j_{s-k}}\cdot \hv)&=\pi ((i_{s-k-1},j_{s-k-1}+1),(i_{s-k},j_{s-k})),\\
    &=f_{i_{s-k-1},j_{s-k-1}}f_{i_{s-k},j_{s-k}}\cdot \hv.
\end{align*} 
 (ii) if $\overline{n-1} \leq j_{s-k-1}\leq \overline{1}$, then:
\begin{align*}
    \pi \circ \phi (f_{i_{s-k-1},j_{s-k-1}}f_{i_{s-k},j_{s-k}}\cdot \hv)&=\pi ((i_{s-k-1},j_{s-k-1}),(i_{s-k},j_{s-k})),\\
    &=f_{i_{s-k-1},j_{s-k-1}}f_{i_{s-k},j_{s-k}}\cdot \hv.
\end{align*}
So, we have $\pi \circ \phi(f^{\p}\cdot \hv)=f^{\p}\cdot \hv$, which completes the proof.\\

Now we are left with showing that this one-to-one correspondence is weight preserving. For this we need to only show that the map:
\begin{align*}
 \phi: \textsc{SyP}_{\lambda} \longrightarrow \textsc{SyST}_{\lambda},\,\,
 f^{\p} \cdot \hv \longmapsto f^{\p} \cdot t_{\lambda},
\end{align*}
is weight preserving, i.e., that $\textsf{wt}(\phi(f^{\p} \cdot \hv))=\textsf{wt}(f^{\p} \cdot \hv)$. For this, we have:
\[\textsf{wt}(\phi(f^{\p} \cdot \hv)) = \textsf{wt}(f^{\p} \cdot t_{\lambda})
    = \textsf{wt}(f^{\p}) + \textsf{wt}(t_{\lambda})
    =\textsf{wt}(f^{\p}) + \textsf{wt}(\hv)
    =\textsf{wt}(f^{\p} \cdot \hv).\] 
\end{proof}

\subsection{A Comparison with Other Existing Tableaux} \label{subsec:comparison}
The PBW-semistandard tableaux of type ${\tt A}$ are defined as follows:
\begin{definition}[\cite{Fei12}]
A \emph{type ${\tt A}$ PBW-semistandard tableau} of shape $\lambda=(\lambda_1\geq \cdots\geq \lambda_{n} \geq 0)$ on the set $\N$ is a filling of the Young diagram $Y_{\lambda}$ with numbers $\T_{i,j}\in \N$ satisfying the properties:
\begin{enumerate}
    \item[(i)] if $\T_{i,j} \leq \mu_j$, then $\T_{i,j}=i$,
    \item[(ii)] if $i_1<i_2$ and $\T_{i_1,j}\neq i_1 $, then $\T_{i_1,j}>\T_{i_2,j}$,
    \item[(iii)] for any $j>1$ and any $i$ there exists $i'\geq i$ such that $\T_{i',j-1}\geq \T_{i,j}$.
\end{enumerate}
\end{definition}
It follows that a symplectic PBW-semistandard tableau is a PBW-semistandard tableau of type {\tt A} which satisfies one extra condition on the columns (condition (iii) of Definition \ref{eqn:symsspbw}).

\begin{example}
For $\g$ of type {\tt A}$_3$, the full set of  PBW-semistandard tableaux restricted to $\lambda = \omega_1 + \omega_2$ ($\lambda=(2,1)$) on the set $\N=\{1,2,\overline{2},\overline{1}\}$ is the one given below:\\
$$
\small\begin{ytableau}
1 & 1\\
2
\end{ytableau},
\begin{ytableau}
1 & 2\\
2
\end{ytableau},
\begin{ytableau}
1 & 1\\
\overline{2}
\end{ytableau},
\begin{ytableau}
1 & 2\\
\overline{2}
\end{ytableau},
\begin{ytableau}
1 & \overline{2}\\
\overline{2}
\end{ytableau},
\begin{ytableau}
\overline{2} & 1\\
2
\end{ytableau},
\begin{ytableau}
\overline{2} & 2\\
2
\end{ytableau},
\begin{ytableau}
\overline{2} & \overline{2}\\
2
\end{ytableau},
\begin{ytableau}
\overline{1} & 1\\
2
\end{ytableau},
\begin{ytableau}
\overline{1} & 2\\
2
\end{ytableau},
$$
$$
\small\begin{ytableau}
\overline{1} & \overline{2}\\
2
\end{ytableau},
\begin{ytableau}
\overline{1} & \overline{1}\\
2
\end{ytableau},
\begin{ytableau}
\overline{1} & 1\\
\overline{2}
\end{ytableau},
\begin{ytableau}
\overline{1} & 2\\
\overline{2}
\end{ytableau},
\begin{ytableau}
\overline{1} & \overline{2}\\
\overline{2}
\end{ytableau},
\begin{ytableau}
\overline{1} & \overline{1}\\
\overline{2}
\end{ytableau},
\begin{ytableau}
1 & 1\\
\overline{1}
\end{ytableau},
\begin{ytableau}
1 & 2\\
\overline{1}
\end{ytableau},
\begin{ytableau}
1 & \overline{2}\\
\overline{1}
\end{ytableau},
\begin{ytableau}
1 & \overline{1}\\
\overline{1}
\end{ytableau}.
$$
When we consider condition (iii) of Definition \ref{eqn:symsspbw}, then we have to drop the last four tableaux from the above list. This way, we are able to recover all the 16 PBW-semistandard tableaux corresponding to $\lambda =\omega_1 +\omega_2$ for $\g$ of type {\tt C}$_{2}$ as seen in Example \ref{exa:tableaux}.
\end{example}
As will be seen in the following section, the symplectic standard tableaux of De Concini in \cite{Dec79} are different from the symplectic PBW-semistandard tableaux because a different condition is imposed on the columns. Furthermore, the symplectic semistandard tableaux of Hamel and King \cite{HK11}, King \cite{Kin76}, Kashiwara and Nakashima \cite{KN91} and Proctor \cite{Pro90} yield semistandard Young tableaux when restricted to type ${\tt A}_{n-1}$, i.e., if entries are taken from the set $\{1,\ldots,n\}$. Hence they are different from the symplectic PBW-semistandard tableaux since the restriction of these in the same way does not yield semistandard Young tableaux. Notice however that there exist weight preserving bijections between all these symplectic tableaux.

\section{The Complete Symplectic Flag Variety; Symplectic Relations and a Basis for the Coordinate Ring} \label{sec:three}
 In this section, we describe the complete symplectic flag variety together with its defining ideal, and we show that the symplectic PBW-semistandard tableaux label a basis for the multi-homogeneous coordinate ring.

\subsection{Flag Varieties; a Brief Description}\label{subsec:flagvarieties}
 Let $\G$ be a simple, simply connected algebraic group over the field $\mathbb{C}$ with the corresponding Lie algebra $\g$. As before, we have a Cartan decomposition $\g = \n^{+}\oplus  \h \oplus \n^{-}$. We know that $\V_{\lambda}$ has a structure of a $\G$-module with highest weight vector $\hv$. Hence we have an action of $\G$ on the projectivization $\mathbb{P}(\V_{\lambda})$. The \emph{flag variety} $\F$ is the closure of the $\G$-orbit through the highest weight line: 
 \begin{equation}\label{eqn:flagvariety}
     \F=\overline{\G[\hv]} \hookrightarrow{} \mathbb{P}(\V_{\lambda}).
 \end{equation}
 Let $\lambda$ be any dominant integral weight of $\g$. Assume $(\lambda, \alpha_i^{\vee})=0$ if and only if $f_{\alpha_i}$ belongs to $\liep$, the Lie algebra corresponding to $\parabolic$, a parabolic subgroup of $\G$. Then each projective variety $\F$ is as well isomorphic to the quotient $\G/{\parabolic}$ of $\G$ by the parabolic subgroup $\parabolic$ leaving $\mathbb{C}\hv$ invariant. This is the \emph{generalized/partial flag variety}. In particular, when $\lambda$ is regular, the flag variety $\F$ is isomorphic to $\G/{\borel}$, as projective varieties, where $\borel$ is a Borel subgroup, and this is then called the \emph{complete/full flag variety}.

\subsection{The Complete Symplectic Flag Variety; General Description}
Now we consider $\G=\group(\mathbb{C})$. Recall the $2n$-dimensional vector space $\W$ over $\mathbb{C}$ and the fixed basis $\{\w_1,\ldots,\w_{2n}\}$. Recall also the non-degenerate symplectic form $\langle\,\,,\,\,\rangle$, defined by:\\
$\langle \w_i \,\,,\w_{\overline{i}}\,\,\rangle=1\quad \text{for} \quad 1\leq i\leq n \quad \text{and}$ 
$\quad \langle \w_i \,\,,\w_j\,\,\rangle=0\quad \text{for all} \quad 1\leq i,j \leq n, j\neq \overline{i}$,\\
where as before, $\overline{i}=2n+1-i$. The matrix of this symplectic form is given by
  $$\Psi:= \begin{pmatrix} 
        0 & \I_n\\-\I_n &0
  \end{pmatrix},$$ 
where $\I_n$ is the $n\times n$ matrix with 1's along the anti-diagonal and zeros elsewhere.
Recall that an \emph{isotropic subspace} of a symplectic vector space is a subspace on which the symplectic form identically vanishes. For $\W$ as above, all the isotropic subspaces have dimension at most $n$. For $1\leq k\leq n$, the \emph{symplectic Grassmannian} $\grass(k,2n)$ is the quotient of $\group$ by a maximal parabolic subgroup and it is known to coincide with the variety of isotropic $k$-dimensional subspaces of $\W$. Notice that when we drop the isotropic condition, we recover the usual \emph{Grassmannian} which will be denoted by $\Gr(k,2n)$.\\
 
We consider the flag variety $\group/\borel$, where $\borel$ is a Borel subgroup. This is the \emph{complete symplectic flag variety} which we denote by $\complete$ and it coincides with the variety whose points are the full flags
$$\{\vectorspace _1\subset\cdots\subset \vectorspace_n, \quad \dim \quad \vectorspace_k =k\}$$ with $\vectorspace_k\in  \grass(k,2n)$ for all $1\leq k \leq n$. This variety is also referred to as the \emph{isotropic flag variety} as in \cite{Dec79}. 
 \subsection{A Pl\"ucker Embedding for the Symplectic Grassmannian}\label{subsec:plueckerembedding}
Consider the irreducible fundamental $\group$-module $\V_{\omega_k}$ of highest weight $\omega_k$. Let $\nu_{\omega_k}$ denote a highest weight vector in $\V_{\omega_k}$.
We have the standard representation $\V_{\omega_1}\simeq \mathbb{C}^{2n}$ and the canonical embedding, 
$$\V_{\omega_k} \hookrightarrow \bigwedge^k \mathbb{C}^{2n}, \quad \nu_{\omega_k}\mapsto \w_{1}\wedge\cdots\wedge \w_{k}.$$
Notice that unlike in the case of type ${\tt A}$, the image of $\V_{\omega_k}$ under the above embedding is not the whole of $\bigwedge^k \mathbb{C}^{2n}$. This also implies that the projectivization $\mathbb{P}(\V_{\omega_k})$ of $\V_{\omega_k}$ does not coincide with the whole projective space $\mathbb{P}\Big(\bigwedge^k \mathbb{C}^{2n}\Big)$. Nonetheless, we choose an embedding of $\grass(k,2n)$ into the latter space as described below. This is because we want to keep all the Pl\"ucker coordinates in the defining relations for easy comparison with the type {\tt A} case.\\

Let  $\vectorspace_k$ denote an element in $\Gr(k,2n)$ such that $\vectorspace_k=\text{span}(u_{1},\ldots,u_{k})$, for some $u_1,\ldots,u_k\in \mathbb{C}^{2n}$. Consider the Pl\"ucker embedding,
\[\Gr(k,2n) \hookrightarrow{} \mathbb{P}\Big(\bigwedge^k \mathbb{C}^{2n}\Big), \quad \text{span}(u_{1},\ldots,u_{k})\mapsto [u_{1}\wedge\cdots\wedge u_{k}].\] 
For a sequence $\Jm=(j_1<\cdots<j_k) \subset \{1<\cdots<n <\overline{n}<\cdots<\overline{1}\}$, let $\X_{\Jm} := (\w_{j_1}\wedge\cdots\wedge \w_{j_k})^*\in (\bigwedge^k \mathbb{C}^{2n})^*$ be the Pl\"ucker coordinate labelled by $\Jm$. These Pl\"ucker coordinates are $k\times k$ minors of the $2n\times k$ matrix representing the subspace $\vectorspace_k$. The image of $\Gr(k,2n)$ under the above embedding is identified with the set of all Pl\"ucker vectors in $\mathbb{P}^{\binom{2n}{k}-1}$, which are vectors whose coordinates are the $\binom{2n}{k}$ Pl\"ucker coordinates $\X_{\J}$.

Now we consider the embedding of $\grass(k,2n)$ into the Grassmannian $\Gr(k,2n)$, i.e.,
\[\grass(k,2n) \hookrightarrow{}\Gr(k,2n) \hookrightarrow{} \mathbb{P}\Big(\bigwedge^k \mathbb{C}^{2n}\Big).\]
The Pl\"ucker embedding of $\grass(k,2n)$ that we are considering is the composition of the above embeddings. It turns out that the isotropic condition on the elements $\vectorspace_k\in \grass(k,2n)$ leads to linear relations among the $\binom{2n}{k}$ Pl\"ucker coordinates $\X_{\J}$ on $\Gr(k,2n)$. This follows from the work of De Concini \cite{Dec79}, and it implies that these linear relations cut out the image of $\grass(k,2n)$ from $\Gr(k,2n)$ as discussed in the following two subsections.

\subsection{Reverse-admissible Minors and their Correspondence with the Symplectic PBW Tableaux}\label{subsec:reverseminors}
Following \cite{Dec79}, we consider now the variety $\variety$ whose points over $\mathbb{C}$ are the $m$-tuples $(v_1,\ldots,v_m)$ of vectors in $\W$ such that $\langle v_i \,\,,v_j\,\,\rangle=0\quad \text{for all} \quad 1\leq i,j \leq m$, where  $\langle\,\,,\,\,\rangle$ is the symplectic form defined above. The variety $\variety$ is therefore equivalently the variety of $2n\times m$ matrices $\M$ with coefficients in $\mathbb{C}$ such that $\M^T \Psi \M=0$, where $\Psi$ is the matrix of the form $\langle\,\,,\,\,\rangle$ and $\M^T$ denotes the transpose of the matrix $\M$.\\

Denote by $A$ the homogeneous coordinate ring of $\variety$. For $1\leq k \leq m$, let
\begin{equation}\label{eqn:formone}
    (i_k,\ldots,i_1 | j_1,\ldots,j_k)
\end{equation}
be the $k\times k$ minor of the matrix $\M$ where $(i_1,\ldots,i_k)$ are the row indices while $(j_1,\ldots,j_k)$ are the column indices. Therefore we have $1\leq i_1,\ldots,i_k \leq \overline{1}$ and $1\leq j_1,\ldots,j_k \leq m$. For what will follow, let us recall a partial ordering $\leq$ on the subsequences of  $\{1,\ldots,n\}$ of equal length $k$ as follows. Given two such subsequences $H=\{h_1<\cdots<h_k\}$ and $J=\{j_1<\cdots<j_k\}$, we say that $H\leq J$ if $h_1\leq j_1,\ldots,h_k\leq j_k$ with equality if and only if $h_1= j_1,\ldots,h_k= j_k$.\\
  
Let $\I_1, \I_2\subset \{1,\ldots,n\}$ be such that $\I_1 :=\{x_1,\ldots,x_t\}$ and $\I_2 :=\{y_1,\ldots,y_{k-t}\}$ for some $0\leq t\leq k$. Let $\Gamma :=\I_1\cap \I_2=\{\gamma_1,\ldots,\gamma_{\lambda}\}$. Define $\tilde{\I}_1 :=\I_1\backslash \Gamma = \{a_1,\ldots,a_{t-\lambda}\}$ and $\tilde{\I}_2 :=\I_2\backslash \Gamma = \{b_1,\ldots,b_{k-t-\lambda}\}$. The following formula provides useful notation for minors of the form \eqref{eqn:formone}.
 \begin{equation}\label{eqn:minor-expression}
   \boxed{(\I_2,\I_1 | j_1,\ldots,j_k) :=  (\overline{b}_1,\ldots,\overline{b}_{k-t-\lambda},a_{t-\lambda},\ldots,a_1,\overline{\gamma}_{\lambda},\gamma_{\lambda},\ldots,\overline{\gamma}_1, \gamma_1 | j_1,\ldots,j_k).}
\end{equation} 
We call the minor on the right hand side of Equation \eqref{eqn:minor-expression}, the \textit{computed minor} corresponding to $(\I_2,\I_1 | j_1,\ldots,j_k)$. In other words, $\I_1$ corresponds to entries in $\{1,\ldots,n\}$ and $\I_2$ corresponds to entries in $\{\overline{n},\ldots,\overline{1}\}$.

A minor $(\I_2,\I_1 | j_1,\ldots,j_k)$ is called \emph{admissible} if there exists a subset $\Tset\subset $ $\{1,\ldots,n\}\backslash (\I_1\cup \I_2)$ with $|\Tset|=|\Gamma|$ and $\Tset > \Gamma$.

We recall the following result:
\begin{proposition}[\cite{Dec79}]\label{pro:admissible-minors}
In the ring $A$, the coordinate ring of the variety $\variety$, any minor can be expressed as a linear combination of admissible minors of the same size and involving the same columns.
\end{proposition}

To find a connection of the variety $\variety$ to $\complete$, the complete symplectic flag variety, we recall a few more results from \cite{Dec79}. The \emph{isotropic Stiefel variety} $\stiefel_{m,n}$ is the open set in $\variety$ whose points over $\mathbb{C}$ are the $m$-tuples of vectors $(v_1,\ldots,v_m)$ in $\W$ such that $(v_1,\ldots,v_m)$ span an isotropic subspace of $\W$ of dimension equal to min$(n,m)$.

\begin{proposition}[\cite{Dec79}]
 The complement of $\stiefel_{m,n}$ in $\variety$ has codimension $\geq 2$. 
\end{proposition}

\begin{corollary}[\cite{Dec79}] \label{cor:stiefel}
Let $A'$ be the ring of global polynomial functions on $\stiefel_{m,n}$, then $A'=A$, where $A$ is the coordinate ring of $\variety$.
\end{corollary}

Also there is a natural morphism $g: \stiefel_{n,n} \rightarrow \complete$ given by $$g((v_1,\ldots,v_n))=\{\vectorspace_{(v_1)}\subset\vectorspace_{(v_1,v_2)}\subset \cdots\subset \vectorspace_{(v_1,\ldots,v_n)}\},$$ where $\vectorspace_{(v_1,\ldots,v_t)}=\{\text{linear span of}\quad v_1,\ldots,v_t\}$ for some $t$ vectors $v_1,\ldots,v_t$ in $\W$.

\begin{proposition}[\cite{Dec79}] \label{pro:B-bundle}
 The morphism $g: \stiefel_{n,n} \rightarrow \complete$ is a principal $\borel$ bundle, where $\borel$ is the Borel subgroup of upper triangular elements in $\G l(n)$.
\end{proposition}

 Proposition  \ref{pro:B-bundle} implies that we actually have $\complete = \stiefel_{n,n}/ \borel$. This and Corollary  \ref{cor:stiefel} imply that $\mathbb{C}[\complete]$ is a sub-ring of $A$, i.e., it is the ring of invariants in $A$ under the group action of $\borel$ on $\stiefel_{n,n}$. \emph{Right canonical minors} are those with $i$'s on the $i$-th column i.e., minors of the form $(i_k,\ldots,i_1|1,\ldots,k)$. These are all we need to work with in $\mathbb{C}[\complete]$ (see \cite[Theorem 4.8]{Dec79}). We will therefore restrict to these minors, in that we will write $(i_k,\ldots,i_1)$ instead of $(i_k,\ldots,i_1|1,\ldots,k)$ and $(\I_2,\I_1)$ instead of $(\I_2,\I_1|1,\ldots,k)$. We will also write `minor' instead of `right canonical minor' for brevity.\\ 
 
Now we would like to find a connection of the admissible minors to the symplectic PBW tableaux. For this, we choose an equivalent but different version of these minors, which we call \emph{reverse-admissible}. In this regard, keeping the same notation as above, we would like to give the following definition. 

\begin{definition}\label{def:reverse-admissible}
A $k\times k$ minor $(\I_2,\I_1)$ is called \textit{reverse-admissible} if there exists a subset $\Tset\subset \{1,\ldots,n\}\backslash (\I_1\cup \I_2)$ with $|\Tset|=|\Gamma|$ and $\Tset < \Gamma$.
\end{definition}
It turns out that Proposition \ref{pro:admissible-minors} also holds for reverse-admissible minors:
\begin{proposition}\label{pro:reversed-admissible-minors}
 In the ring $\mathbb{C}[\complete]$, any minor can be expressed as a linear combination of reverse-admissible minors of the same size and involving the same columns.
\end{proposition}

 To prove this proposition, we first recall Proposition 1.8 of \cite{Dec79}, and a modified version of Definition 1.4 of \cite{Dec79} which gives a total ordering on the set of right canonical minors.

\begin{proposition}[\cite{Dec79}]\label{pro:helper}
Let $(\tilde{\I}_2\cup \Gamma, \tilde{\I}_1\cup \Gamma)$ be a fixed minor of size $k\leq n$, Then on $\complete$, the following relations hold.
\begin{equation}\label{eqn:admissible-relations}
  (\tilde{\I}_2\cup \Gamma, \tilde{\I}_1\cup \Gamma)=
  \sum_{\Gamma': |\Gamma'|=|\Gamma|\,\, \textsf{and}\,\, \Gamma'\cap \{\tilde{\I}_1 \cup \tilde{\I}_2\cup \Gamma\}=\emptyset}(-1)^{|\Gamma'|}(\tilde{\I}_2\cup \Gamma', \tilde{\I}_1\cup \Gamma').
\end{equation}
\end{proposition}

\begin{definition}\label{def:total-ordering}
Given two $k\times k$ minors $\Lm=(l_1,\ldots,l_k)$ and $\Jm=(j_1,\ldots,j_k)$, we say that $\Lm \unlhd \Jm$ if $\nu_{\Lm}:=(l_1+\dots+l_k)< (j_1+\cdots+j_k) =:\nu_{\Jm}$ or $\nu_{\Lm}= \nu_{\Jm}$, and the last non zero entry of the vector $\Lm-\Jm$ is positive.
\end{definition}

We are now set to prove Proposition \ref{pro:reversed-admissible-minors}:

\begin{proof}[Proof of Proposition~\ref{pro:reversed-admissible-minors}] The proof is in principle similar to the proof of Proposition 2.2 of \cite{Dec79}. We will therefore adapt the same proof here. Consider a minor $(\I_2,\I_1)$ which is not reverse-admissible. We will show that $(\I_2,\I_1)$ can be written as a linear combination of minors of the same size that are smaller in the ordering $\unlhd$ of Definition \ref{def:total-ordering}. Clearly, we only need to consider the case $\Gamma = \I_1 \cap \I_2 \neq \emptyset$. Now let $\Gamma = \{\gamma_1,\ldots,\gamma_t\}$. Choose $1\leq h_0\leq t$ minimally such that there exists a tuple $\Tset\subset \{1,\dots,n\}\setminus (\I_1\cup \I_2)$ of length $t-h_0$ with 
$$ \Tset < (\gamma_{h_0+1},\dots,\gamma_t).$$ Choose $\Tset_{h_0+1}=\{\lambda_{h_0+1},\dots,\lambda_t\}$ maximal (with respect to the partial order $\leq$) among those $\Tset$. Choose $b\in\{h_0+1,\dots,t\}$ maximally such that $$(\lambda_{h_0+1},\dots,\lambda_b) < (\gamma_{h_0},\dots,\gamma_{b-1}),
$$
or set $b=h_0$ if no such $b$ exists. Now define $\tilde\Gamma:=(\gamma_{h_0},\dots,\gamma_b)$. Recall the subsets of $\{1,\ldots,n\}$;  $\tilde{\I}_1=\I_1\setminus \Gamma$ and $\tilde{\I}_2=\I_2\setminus \Gamma$. 
 Applying Relation \eqref{eqn:admissible-relations} from Proposition \ref{pro:helper} to $\tilde{\Gamma}$ and taking $\mathsf{F}=\Gamma\backslash \tilde{\Gamma}$, we find:\\
\begin{equation}\label{eqn:minor-summation}
    (\I_2,\I_1)=(-1)^{b-h_0+1}\sum_{\Gamma': \Gamma'\cap \{\I_1\cup\I_2\}=\emptyset}(\tilde{\I}_2\cup\mathsf{F}\cup \Gamma',\tilde{\I}_1\cup\mathsf{F}\cup \Gamma'),
\end{equation}
with $|\tilde{\Gamma}|=|\Gamma'|$. For any $\Gamma'=\{\gamma_{h_0}'<\cdots<\gamma_{b}'\}$ appearing on the right-hand side of \eqref{eqn:minor-summation}, the sum $\nu$ defined in Definition \ref{def:total-ordering} has the same value which it takes for $(\I_2,\I_1)$. We claim now that for every such $\Gamma'$, we have $\gamma_b' > \gamma_b$. We will assume the contrary that $\gamma_b' \leq \gamma_b$. Now since $\gamma_b'\subset \{1,\ldots,n\}\backslash (\I_1\cup \I_2)$ and $\lambda_{b+1}>\gamma_b$ (by the maximality of $b$), the maximality of $\Tset_{h_0+1}$ implies $\gamma_b'\leq \lambda_b$. Now suppose by induction that $\gamma_e'\leq \lambda_e$, for all $ h_0+1<f \leq e \leq b$, then $\gamma_{f-1}'<\gamma_{f}'\leq \lambda_f<\gamma_{f-1},$ and if $f-1\leq h_0+1$, the maximality of $\Tset_{h_0+1}$ implies that $\gamma_{f+1}'<\lambda_{f+1}$, so that if $f-1=h_0$ we have $\gamma_{h_0}'<\gamma_{h_0}$. In particular $\gamma_e'<\gamma_e$  for all $h_0\leq e \leq b$. This then implies that we have $$(\gamma_{h_0}',\ldots,\gamma_b',\lambda_{b+1},\ldots,\lambda_t)< (\gamma_{h_0},\ldots,\gamma_t)$$ component-wise and $\{\gamma_{h_0}',\ldots,\gamma_b',\lambda_{b+1},\ldots,\lambda_t\}\subset\{1,\ldots,n\}\backslash (\I_1\cup \I_2)$, which contradicts the minimality of $h_0$. Thus we have $\gamma_b' > \gamma_b$ and this together with what has been noted about the sum $\nu$, implies that each minor appearing on the right-hand side of \eqref{eqn:minor-summation} is smaller than $(\I_2,\I_1)$ in the ordering $\unlhd$, which proves the proposition.
\end{proof}

\begin{example}
Consider $n=4$, $k=4$, $\I_1=\{1,2\}$ and $\I_2=\{1,2\}$. Then $\{1,2,3,4\}\backslash \{\I_1 \cup \I_2\}=\{3,4\}$. The minor $(\I_2,\I_1)$ is not reverse-admissible in the sense of Definition \ref{def:reverse-admissible}. We have $\Gamma =\{1,2\}$, so $h_0=2$. With this, we get $b=2$, giving us $\tilde{\Gamma}=\{2\}$,  Moreover we have $\tilde{\I}_1=\tilde{\I}_2=\emptyset$. Also $\mathsf{F} = \Gamma\backslash\tilde{\Gamma}=\{1\}$. So substituting into Equation \eqref{eqn:minor-summation}, we get:
\begin{align}\label{minor-example-one}
    (\{1,2\},\{1,2\})&=(-1)^1\Big[(\emptyset \cup \{1\} \cup \{3\}, \emptyset \cup \{1\} \cup \{3\})\nonumber\\
    &+ (\emptyset \cup \{1\} \cup \{4\}, \emptyset \cup \{1\} \cup \{4\})\Big],\nonumber\\
    &= -(\{1,3\},\{1,3\}) - (\{1,4\},\{1,4\}).
\end{align}
Computing all the minors according to Equation \eqref{eqn:minor-expression}, Equation \eqref{minor-example-one} above becomes:
\begin{equation*}
  (\overline{2},2,\overline{1},1)  = -(\overline{3},3,\overline{1},1) - (\overline{4},4,\overline{1},1).
\end{equation*}
\end{example}

From a $k\times k$ minor $(\I_2,\I_1)$ (which is not necessarily reverse-admissible), we describe how to obtain a single column length $k$ PBW tableau (which is not necessarily symplectic). For this, we first compute the minor $(\I_2,\I_1)$ according to Equation \eqref{eqn:minor-expression}. We then put every entry which is less than or equal to $k$ at its position in the column of length $k$, and every other entry should be put in such a way that it is bigger than entries below it. For example, the single column length 4 PBW tableaux corresponding to the $4\times 4$ computed minors $(\overline{2},2,\overline{1},1)$,  $(\overline{3},3,\overline{1},1)$, and  $(\overline{4},4,\overline{1},1)$ are respectively the tableaux:
$$\small\begin{ytableau}
1\\2\\ \overline{1}\\ \overline{2}
\end{ytableau}, \,\,\,\begin{ytableau}
1\\\overline{1}\\3\\ \overline{3}
\end{ytableau},\,\,\, \text{and} \,\,\, \small\begin{ytableau}
1\\\overline{1}\\ \overline{4}\\4
\end{ytableau}.$$

Moreover, we can also recover a $k\times k$ minor $(\I_2,\I_1)$ from the corresponding single column length $k$ PBW tableau. To do this, we put every element $i$ for which $\overline{i}$ belongs to our tableau in $\I_2$, and all other elements which appear in the tableau without bars are put in $\I_1$. We prove the following result.

\begin{proposition}\label{pro:minors-tableaux}
The reverse-admissible $k\times k$ minors are in a weight preserving bijection with the single column length k symplectic PBW tableaux.
\end{proposition}
\begin{proof}
We first show that the tableaux corresponding to the reverse-admissible minors $(\I_2,\I_1)$ satisfy the conditions of Definition \ref{def:fundamentals}. Conditions (i) and (ii) are clearly always satisfied up to a reordering of the entries appearing in the computed minor of $(\I_2,\I_1)$ as described above. It remains to verify condition (iii) of Definition \ref{def:fundamentals}, namely, we want to show that whenever we have a pair $(i,\overline{i})$ with $i< k$ in the computed minor, then after re-ordering to satisfy (i) and (ii), the position of $\overline{i}$ is above that of $i$. 

Recall that if $\Gamma = \I_1 \cap \I_2 = \{\gamma_1,\ldots,\gamma_{\lambda}\}$, then we have that $(\I_2,\I_1)$ is reverse-admissible if we can find $\Tset\subset \{1,\ldots,n\}\backslash (\I_1\cup \I_2)$ with  $\Tset=\{\nu_1,\ldots,\nu_{\lambda}\}$ (i.e., $|\Tset|=|\Gamma|$), and $\nu_1\leq \gamma_1, \ldots,\nu_{\lambda}\leq \gamma_{\lambda}$. Consider the computed minor of $(\I_2,\I_1)$, $\Lm = (\overline{b}_1,\ldots,\overline{b}_{k-t-\lambda},a_{t-\lambda},\ldots,a_1,\overline{\gamma}_{\lambda},\gamma_{\lambda},\ldots,\overline{\gamma}_1, \gamma_1)$. We are going to describe how to fill in a column. We put each $\overline{\gamma}_i$ at position $\nu_i$, each $\gamma_i$ at position $\gamma_i$, each $a_i$ at position $a_i$, and the $\overline{b}_i$'s at the remaining spots in a descending order from top to bottom. This implies that $\overline{\gamma}_i$ is above $\gamma_i$ since $\nu_i\leq \gamma_i$ for all $1\leq i\leq k$ and $\Tset \cap \tilde{\I}_1 = \emptyset$. Hence the resulting column is a single column symplectic PBW tableau. 

For the other direction, assume we are given a single column symplectic PBW tableau. For all $\overline{i}$ in the tableau, put $i$ in $\I_2$, and put the rest of the indices in $\I_1$. Also, for all $(i_1,\ldots,i_{\lambda})$ for which we have $(\overline{i}_1,\ldots,\overline{i}_{\lambda})$ in the column, let $j_t$ be the position of $\overline{i}_t$ for all $1\leq t \leq \lambda$. The tableau being a symplectic PBW tableau implies $j_t\leq i_t$ for all $1\leq t \leq \lambda$. Also, we note that $j_t\in \{1,\ldots,n\}\backslash (\I_1\cup \I_2)$ for all $1\leq t \leq \lambda$ and hence the set $\{j_1,\ldots,j_{\lambda}\}$ is the minimal set with the required properties. Hence $(\I_2,\I_1)$ is reverse-admissible. This gives the bijection. The fact that this bijection is weight preserving is straight forward. 
\end{proof}

Proposition \ref{pro:minors-tableaux} allows us to use the notions reverse-admissible minors and single column symplectic PBW tableaux interchangeably. We do this henceforth.  

\subsection{Defining Ideal}
Consider the embeddings
\begin{equation}\label{eqn:classical-embedding}
    \complete \hookrightarrow\prod_{k=1}^n\grass(k,2n)\hookrightarrow\prod_{k=1}^n\Gr(k,2n)\hookrightarrow \prod_{k=1}^n \mathbb{P}\Big(\bigwedge^k \mathbb{C}^{2n}\Big).
\end{equation}
The composition of the above embeddings is the Pl\"ucker embedding of $\complete$ that we are considering. Consider the polynomial ring $\mathbb{C}[\X_{j_1,\ldots,j_d}]$ generated by all the elements $\X_{j_1,\ldots,j_d}$, $d=1,\ldots,n$ and $1\leq j_1<\cdots<j_d\leq \overline{1}$. We want to describe the defining (or vanishing) ideal in $\mathbb{C}[\X_{j_1,\ldots,j_d}]$ for $\complete$ under the composition of the above embeddings. We first recall how this ideal relates to that of the type {\tt A} flag variety in the following remark.

\begin{remark}\label{rem:inclusion}
Consider $\g=\mathfrak{sl}_{2n}(\mathbb{C})$ and $\G=\textsc{SL}_{2n}(\mathbb{C})$ in the construction of Subsection \ref{subsec:flagvarieties}. Let $1\leq d_1<\cdots<d_s\leq 2n-1$ be a sequence of increasing numbers. For an $\mathfrak{sl}_{2n}$ dominant integral weight $\lambda=\sum_{s=1}^{2n-1}m_s\omega_{d_s}$, we consider the corresponding \emph{partial flag variety}, which we denote by $\textsc{SL}\mathcal{F}_{\lambda}$. This variety is known to coincide with the projective variety of partial flags
$$\{\vectorspace _1\subset\cdots\subset \vectorspace_s: \,\,\,\, \dim\,\, \vectorspace_k =k\},$$
where $\vectorspace_k \subset \mathbb{C}^{2n}$ for $1\leq k\leq s$. We also consider the Pl\"ucker embedding of this variety.  

For an algebraic variety $X$, let  $\ideal(X)$ denote its vanishing ideal. Then for any $\lambda$ of the form $\lambda= \sum_{s=1}^n m_s\omega_s$, with all $m_s\ne 0$, we have the inclusion $\ideal(\textsc{SL}\mathcal{F}_{\lambda})\subset\ideal(\complete)$, since $\complete$ is point-wise contained in $\textsc{SL}\mathcal{F}_{\lambda}$ for such $\lambda$. 
\end{remark}

\begin{definition}\label{def:classicalrelations}
Let $\Lm, \Jm\subset \{1,\ldots,n,\overline{n},\ldots, \overline{1}\}$ be two sequences of length $p$ and $q$ respectively, with $n\geq p\geq q \geq 1$. Suppose $\Lm= \{l_1,\ldots,l_p\}$, with $l_1<\cdots <l_p$ and $\Jm = \{j_1,\ldots,j_q\}$ with $j_1<\cdots<j_q$ after rearrangement. For $1\leq t\leq q$, we have the \emph{Pl\"ucker relation}
\begin{align}\label{eqn:relations-Cn}
   R_{\Lm,\Jm}^t:= \X_{\Lm} \X_{\Jm}- \sum_{1\leq r_1<\cdots<r_t\leq p}\X_{\Lm'}\X_{\Jm'},
\end{align} 
where $\Lm'$ and $\Jm'$ are obtained from $\Lm$ and $\Jm$ by interchanging $t$-tuples $(l_{r_1},\ldots,l_{r_t})$ and $(j_1,\ldots,j_t)$ in $\Lm$ and $\Jm$ respectively, while keeping the order in which they appear. 

Recall the notation of Subsection \ref{subsec:reverseminors}. Consider a non reverse-admissible minor $(\I_2,\I_1)$. Now let $\Gamma =\I_1 \cap \I_2= \{\gamma_1,\ldots,\gamma_t\}$. Let $1\leq h_0\leq t$ be chosen minimally such that there exists a tuple $\Tset\subset \{1,\dots,n\}\setminus (\I_1\cup \I_2)$ of length $t-h_0$ with 
$ \Tset < (\gamma_{h_0+1},\dots,\gamma_t).$ Let $\Tset_{h_0+1}=\{\lambda_{h_0+1},\dots,\lambda_t\}$ be chosen such that it is maximal (with respect to the partial order $\leq$) among those $\Tset$. Let $b\in\{h_0+1,\dots,t\}$ be chosen maximally such that $(\lambda_{h_0+1},\dots,\lambda_b) < (\gamma_{h_0},\dots,\gamma_{b-1})
$, or set $b=h_0$ if no such $b$ exists. Define $\tilde\Gamma:=(\gamma_{h_0},\dots,\gamma_b)$. Recall the sets $\tilde{\I}_1=\I_1\setminus \Gamma$, $\tilde{\I}_2=\I_2\setminus \Gamma$ and $\mathsf{F}=\Gamma\backslash \tilde{\Gamma}$. We have the linear relation
\begin{equation}\label{eqn:linear-relations}
   S_{(\I_2,\I_1)} := \X_{(\I_2,\I_1)}- \sum_{\Gamma':\, \Gamma'\cap\{\I_1\cup\I_2\}=\emptyset\,\, \text{and}\,\, |\Gamma'|=|\tilde{\Gamma}|}(-1)^{|\Gamma'|} \X_{(\tilde{\I}_2\cup\mathsf{F}\cup \Gamma', \tilde{\I}_1\cup\mathsf{F}\cup \Gamma')}.
\end{equation}
We call the linear relation $S_{(\I_2,\I_1)}$ \emph{symplectic relation}.
\end{definition}

\begin{remark}
In both Expressions \eqref{eqn:relations-Cn} and \eqref{eqn:linear-relations}, the following equality is assumed: $$\X_{j_{\sigma(1)},\ldots,j_{\sigma(d)}}=(-1)^{l(\sigma)}\X_{j_1,\ldots,j_d},$$ for all $d=1,\ldots,n$ and $1\leq j_1<\cdots<j_d\leq \overline{1}$, where $l(\sigma)$ is the inversion number of $\sigma\in S_{d}$.
\end{remark}

Let $\Ideal$ denote the ideal of the polynomial ring $\mathbb{C}[\X_{j_1,\ldots,j_d}]$ generated by the Pl\"ucker relations $R_{\Lm,\Jm}^t$ and the symplectic relations $S_{(\I_2,\I_1)}$. We have:

\begin{theorem}[\cite{Dec79}]\label{thm:classical-ideal}
The ideal $\Ideal$ is the defining ideal of $\complete$ with respect to the Pl\"ucker embedding, $\complete \hookrightarrow \prod_{k=1}^n \mathbb{P}\Big(\bigwedge^k \mathbb{C}^{2n}\Big)$. It is a prime ideal.
\end{theorem}
\begin{proof}
By Remark \ref{rem:inclusion}, we have that for any $\lambda$ of the form $\lambda= \sum_{s=1}^n m_s\omega_s$, with all $m_s\ne 0$, the inclusion $\ideal(\textsc{SL}\mathcal{F}_{\lambda})\subset\ideal(\complete)$ of defining ideals holds. It therefore follows that the relations $R_{\Lm,\Jm}^t$ are satisfied on $\complete$, since they are satisfied on the type {\tt A}$_{2n-1}$ partial flag variety $\textsc{SL}\mathcal{F}_{\lambda}$ according to \cite[Lemma 1, p. 132]{Ful97}. The relations $S_{(\I_2,\I_1)}$ come from Equation \eqref{eqn:minor-summation} from the proof of Proposition \ref{pro:reversed-admissible-minors}, so they are clearly satisfied. 

It follows from \cite[Theorem 4.8]{Dec79} that the relations $R_{\Lm,\Jm}^t$ and $S_{(\I_2,\I_1)}$ are enough to express every non standard symplectic tableau as a linear combination of symplectic standard tableaux in the homogeneous coordinate ring of $\complete$. That is to say, these relations provide a straightening law for this coordinate ring. This therefore implies that these relations generate the defining ideal of $\complete$. This ideal is prime since $\complete$ is irreducible. 
\end{proof}
 
\begin{example}
For $\comp_4$, the ideal $\Ideal$ is generated by the following relations:
$$
\begin{array}{lcl}
R_{(1,2),(\overline{2},\overline{1})}^{1}  &:=&\X_{1,2}\X_{\overline{2},\overline{1}}-\X_{1,\overline{2}}\X_{2,\overline{1}}+\X_{1,\overline{1}}\X_{2,\overline{2}},\\
R_{(1,2),(\overline{2})}^{1} &:=&\X_{1,2}\X_{\overline{2}}+\X_{2,\overline{2}}\X_1-\X_{1,\overline{2}}\X_2,\\
R_{(1,\overline{2}),(\overline{1})}^{1} &:=&\X_{1,\overline{2}}\X_{\overline{1}}+\X_{\overline{2},\overline{1}}\X_1-\X_{1,\overline{1}}\X_{\overline{2}},\\
R_{(1,2),(\overline{1})}^{1} &:=&\X_{1,2}\X_{\overline{1}}+\X_{2,\overline{1}}\X_1-\X_{1,\overline{1}}\X_2,\\
R_{(2,\overline{2}),(\overline{1})}^{1}&:=&\X_{2,\overline{2}}\X_{\overline{1}}+\X_{\overline{2},\overline{1}}\X_2-\X_{2,\overline{1}}\X_{\overline{2}},\\
S_{1,\overline{1}}&:=&\X_{1,\overline{1}}+\X_{2,\overline{2}}.
\end{array}
$$
\end{example}

\subsection{A Basis for the Coordinate Ring} Let $\mathbb{C}[\complete]$ denote the multi-homogeneous coordinate ring of $\complete$. One has the following statement: $$\mathbb{C}[\X_{j_1,\ldots,j_d}]/\Ideal = \mathbb{C}[\complete] = \bigoplus_{\lambda\in P^+} \mathbb{C}[\complete]_{\lambda}\simeq \bigoplus_{\lambda\in P^+} \V_{\lambda}^*,$$ where the multiplication $\V_{\lambda}^*\otimes \V_{\mu}^* \rightarrow \V_{\lambda+\mu}^*$ is induced by the injective homomorphism of modules: $\V_{\lambda+\mu} \hookrightarrow{} \V_{\lambda}\otimes \V_{\mu}, \quad \nu_{\lambda + \mu} \mapsto \nu_{\lambda}\otimes \nu_{\mu}$. The isomorphism $\mathbb{C}[\complete]_{\lambda} \simeq \V_{\lambda}^*$ is given by the classical Borel-Weil theorem.\\

We want to prove that the symplectic PBW-semistandard tableaux index a basis for the coordinate ring $\mathbb{C}[\complete]$. Recall the set $\SyST_{\lambda}$ of these tableaux of shape $\lambda=(\lambda_1\geq \cdots\geq \lambda_n\geq 0)$. To an element $\T\in \SyST_{\lambda}$, we associate the monomial:
\begin{equation}\label{eqn:tableauxmonomial}
    \X_{\T}=\prod_{j=1}^{\lambda_1}\X_{\T_{1,j},\ldots,\T_{\mu_j,j}}\in \V_{\lambda}^*.
\end{equation}

We have:
\begin{proposition}\label{thm:final-result}
The elements $\X_{\T}$, $\T\in \SyST_{\lambda}$, form a basis of $\mathbb{C}[\complete]_{\lambda}$.
\end{proposition}

In order to prove this proposition, we first prove a useful result about the PBW-degree introduced in Section \ref{sec:preliminaries}. This will also be important for the constructions in Section \ref{sec:four}. 

The following is the central definition in this light:

\begin{definition}
For any sequence $\Jm = (j_1<\cdots<j_k)\subset\{1,\ldots,n,\overline{n},\ldots,\overline{1}\}$, $1\leq k \leq n$, the PBW-degree of $\Jm$ is given by the formula:
\begin{equation}\label{def:deg-computed-minor}
    \deg\,\, \Jm = \# \{r: j_r>k\}.
\end{equation}
The PBW-degree of the minor $(\I_2,\I_1)$ is the PBW-degree of its computed minor as given in Equation \eqref{eqn:minor-expression}. 
The PBW-degree of an element $\X_{\Jm}$ is defined to be the PBW-degree of $\Jm$. The PBW-degree of an element $\X_{\T}$ from \eqref{eqn:tableauxmonomial} above is the sum of the PBW-degrees of the elements $\X_{\T_{1,j},\ldots,\T_{\mu_j,j}}$ appearing in the product.
\end{definition}

\begin{remark}\label{rem:pbwdegree}
We would like to give an interpretation of the above definition of PBW-degree in terms of the definition given in Section \ref{sec:preliminaries}. Consider $\g=\mathfrak{sl}_{2n}(\mathbb{C})$ with Cartan decomposition $\mathfrak{sl}_{2n}=\n^+_{\mathfrak{sl}_{2n}}\oplus \h_{\mathfrak{sl}_{2n}}\oplus \n^-_{\mathfrak{sl}_{2n}}$. Consider the irreducible fundamental $\mathfrak{sl}_{2n}$-module $L_{\omega_k}$ of highest weight $\omega_k$. We identify $L_{\omega_k}$ with  $\bigwedge^k \mathbb{C}^{2n}$. Recall the standard basis $\{ \w_1, \ldots, \w_{2n}\}$ of $\mathbb{C}^{2n}$. Then a basis of $$\bigwedge^k \mathbb{C}^{2n}$$ is given by the elements $\w_{\Jm}:=\w_{j_1}\wedge\cdots\wedge\w_{j_k}$ labelled by all sequences $\Jm = (j_1<\cdots<j_k)\subset\{1,\ldots,n,\overline{n},\ldots,\overline{1}\}$. Here the highest weight vector $\nu_{\omega_k}$ is identified with $\w_1\wedge\cdots\wedge\w_k$. Recall that the PBW-degree of an element $\w_{j_1}\wedge\cdots\wedge\w_{j_k}\in L_{\omega_k}$, is the smallest positive integer $s$ such that there exists a polynomial $p \in \U(\n^-_{\mathfrak{sl}_{2n}})$ of degree $s$ with 
\[
p. \w_1\wedge\cdots\wedge\w_k=\pm\w_{j_1}\wedge\cdots\wedge\w_{j_k}.
\]
In fact, $p$ can be chosen as $f_{i_1,j_1}\cdots f_{i_s,j_s}$ for some root vectors $f_{i_t, j_t}, \,\, 1\leq t\leq s$. This degree therefore corresponds to the number of indices in $\Jm$ that are greater than $k$. The component $(L_{\omega_k})_s$ is spanned by the elements $\w_{\Jm}$ with $\deg \w_{\Jm}\leq s$. Therefore, the images of the elements $\w_{\Jm}$ with $\deg \w_{\Jm}= s$ give a basis of $(L_{\omega_k})_s/(L_{\omega_k})_{s-1}\xhookrightarrow{} L_{\omega_k}^a$, the PBW degenerate module. Let $\w_{\Jm}^a$ denote these images and let $\X_{\Jm}^a$ denote the dual elements.

Now we consider the irreducible fundamental
$\syp$-module $\V_{\omega_k}\subset L_{\omega_k}\simeq \bigwedge ^k \mathbb{C}^{2n}$. Consider an arbitrary element in $\V_{\omega_k}$ and write it as a sum of pure wedge tensors. Then find the minimal $\mathfrak{sl}_{2n}$-component which contains all the summands. We consider the maximal $\mathfrak{sl}_{2n}$-degree of each of these summands. This is the PBW-degree induced on $\syp$. By the PBW theorem, this induced PBW-degree is compatible with the one on $\mathfrak{sl}_{2n}$ (\cite{BKF21}).
\end{remark}

\begin{remark}
We can obtain the PBW-degree of $(\I_2,\I_1)$ directly from the subsequences $\I_1$ and $\I_2$ without first computing the minor. For this, we use the formula:
\begin{equation}\label{def:deg-not-computed-minor}
    \deg (\I_2,\I_1) = |\I_2|+\# \{i\in \I_1 : i>k\}.
\end{equation}
To see that the degrees given in Equations \eqref{def:deg-computed-minor} and \eqref{def:deg-not-computed-minor} agree, we only need to consider the PBW-degree of the computed minor $\Lm$ of $(\I_2,\I_1)$. Indeed from Equation \eqref{eqn:minor-expression}, we have that:
\begin{align*}
  \deg \Lm &= |\tilde{\I}_2|+|\Gamma|+\#\{z:    a_z \in \tilde{\I}_1, a_z> k\} + \#\{z:    \gamma_z \in \Gamma, \gamma_z> k\},\\
    &=|\I_2\backslash \Gamma|+|\Gamma|+\#\{z:    i_z \in \I_1, i_z> k\}-\#\{z:    \gamma_z \in \Gamma ,\gamma_z> k\}\\
    &+ \#\{z:    \gamma_z \in \Gamma ,\gamma_z> k\},\\
    &=|\I_2|+\#\{z:    i_z \in \I_1, i_z> k\},\\
    &=\deg (\I_2,\I_1).
\end{align*}
\end{remark}
 We prove the following fundamental lemma.
\begin{lemma}\label{lem:PBW-linear}
Following the notation of Definition \ref{def:classicalrelations}, the PBW-degree of each of the summands appearing in the right hand side of
\begin{equation}\label{eqn:PBW-linear}
   \X_{(\I_2,\I_1)}= \sum_{\Gamma':\, \Gamma'\cap\{\I_1\cup\I_2\}=\emptyset\,\, \text{and}\,\, |\Gamma'|=|\tilde{\Gamma}|}(-1)^{|\Gamma'|} \X_{(\tilde{\I}_2\cup\mathsf{F}\cup \Gamma',\tilde{\I}_1\cup\mathsf{F}\cup \Gamma')},
\end{equation}
is greater than or equal to the PBW-degree of the term $\X_{(\I_2,\I_1)}$ on the left hand side, whenever $(\I_2,\I_1)$ is not reverse-admissible.
\end{lemma}
\begin{proof}
We claim that
$|\I_2|=|\tilde{\I}_2\cup\mathsf{F}\cup \Gamma'|$ for every $\Gamma'$. Indeed one has
\begin{equation}\label{eqn1:lemma317}
|\tilde{\I}_2\cup\mathsf{F}\cup \Gamma'|=|\I_2\backslash \Gamma  \cup \Gamma\backslash \tilde{\Gamma}\cup \Gamma'|= |\I_2\backslash \tilde{\Gamma}\cup \Gamma'|=|\I_2|- |\tilde{\Gamma}|+ |\Gamma'|=|\I_2|,
\end{equation}
since $|\tilde{\Gamma}|=|\Gamma'|$. Now we will show that 
\begin{equation}\label{eqn2-lemma317}
    \# \{i\in \tilde{\I}_1\cup\mathsf{F}\cup \Gamma' : i>k\} \geq \# \{i\in \I_1 : i>k\}.
\end{equation} 
But we know that $\# \{i\in \I_1 : i>k\} = \# \{i\in \tilde{\I}_1\cup\mathsf{F}\cup \tilde{\Gamma} : i>k\}$. Therefore proving the Inequality \eqref{eqn2-lemma317} reduces to showing that: 
$\# \{i\in \tilde{\I}_1\cup\mathsf{F}\cup \Gamma' : i>k\} \geq \# \{i\in \tilde{\I}_1\cup\mathsf{F}\cup \tilde{\Gamma} : i>k\},$ which in turn reduces to showing that: 
$\# \{i\in \Gamma' : i>k\} \geq \# \{i\in \tilde{\Gamma} : i>k\}.$
In fact from the proof of Proposition \ref{pro:reversed-admissible-minors}, we know that the maximum element in  $\tilde\Gamma$ is $\gamma_b$. 

We claim that $\gamma_b < k$. For $(\I_2,\I_1)$  non reverse-admissible, recall the set $\Tset =\Tset_{h_0+1}$. Claim: for all $i<\gamma_b$, $i\in \Tset\cup \I_1\cup \I_2$.
Assume this was not true, then $\Tset\cup\{i\}<(\gamma_{h_0},\dots,\gamma_t)$, which contradicts the minimality of $h_0$. Set $M=\{1,\dots,\gamma_b\}$. Then by the claim, and as $\Tset\cap(\I_1\cup \I_2)=\emptyset$,
$$
\gamma_b = |M| = |M\cap \Tset| + |M\cap (\I_1\cup \I_2)|
\leq |\Tset| + |\I_1\cup \I_2| < |\Gamma| + |\I_1\cup \I_2| = k
.
$$
This together with \eqref{eqn1:lemma317} implies the lemma. 
\end{proof}

We end this section by proving Proposition \ref{thm:final-result}.

\begin{proof}[Proof of Proposition~\ref{thm:final-result}] From Theorem \ref{thm:correspondence}, we have that 
\[\#\{\T :\,\, \T\in\SyST_{\lambda}\}=\dim \V_{\lambda},\] so it remains to show that the elements $\X_{\T}$, $\T\in$ $\SyST_{\lambda}$, span $\mathbb{C}[\complete]_{\lambda}$. Suppose we are given an element $\X_{\T}$ with $\T\notin$ $\SyST_{\lambda}$. We claim that $\X_{\T}$ can be written as a linear combination of elements $\X_{\T'}$ with $\T'\in$ $\SyST_{\lambda}$. 

From Propositions \ref{pro:reversed-admissible-minors} and \ref{pro:minors-tableaux}, it suffices to consider only PBW tableaux $\T$ that are symplectic but not PBW-semistandard. Recall that the condition of PBW-semistandardness is defined between neighbouring columns of $\T$. It is therefore sufficient to consider any two such columns of $\T$ which violate this condition. Let these columns be denoted by $\Lm$ and $\Jm$, with length of $\Lm$ equal to $p$ and length of $\Jm$ equal to $q$ such that $p\geq q$.  We first of all apply the Pl\"ucker relation \eqref{eqn:relations-Cn}, to express the product $\X_{\Lm} \X_{\Jm}$ as a sum of products $\X_{\Lm^{(i)}} \X_{\Jm^{(i)}}$, that is:
\begin{align*}
    \X_{\Lm} \X_{\Jm}= \sum_{i}\X_{\Lm^{(i)}} \X_{\Jm^{(i)}}.
\end{align*}

However, after exchanging, it may happen that for one of the variables $\X_{\Lm^{(i)}}$ or $\X_{\Jm^{(i)}}$, the corresponding $\Lm^{(i)}$ or $\Jm^{(i)}$ is no longer a single column symplectic PBW tableau, that is to say, the corresponding minor $(\I_2,\I_1)^{(i)}$ is not reverse-admissible (by Proposition \ref{pro:minors-tableaux}). In this case, we apply the symplectic relation \eqref{eqn:linear-relations}, to replace such a variable with a sum of variables corresponding to reverse-admissible minors. 

Now from Lemma \ref{lem:PBW-linear} and from the proof of Proposition 4.12 of \cite{Fei12}, we see that $$\deg \X_{\Lm^{(i)}}+ \deg \X_{\Jm^{(i)}} \geq \deg \X_{\Lm} + \deg \X_{\Jm}.$$ 
Therefore in $\mathbb{C}[\complete]_{\lambda}$, any $\X_{\T}$ with $\T\notin$ $\SyST_{\lambda}$, can be written as a linear combination of $\X_{\T'}$ with the sum of PBW-degrees of $\X_{\T'}$ bigger than or equal to that of $\X_{\T}$. This implies the claim since the sum of PBW-degrees of fixed shaped tableaux is bounded from above.

\end{proof}

\section{The Complete Symplectic PBW Degenerate Flag Variety; Symplectic Degenerate Relations and a Basis for the Coordinate Ring} \label{sec:four} 
In this section, we describe the complete symplectic PBW degenerate flag variety following \cite{FFL14}. We then prove results on a basis for the multi-homogeneous coordinate ring and the generating set of relations for the defining ideal with respect to the Pl\"ucker embedding.

\subsection{PBW Degenerate Flag Varieties; a Brief  Description}\label{subsec:deg} Let $\G^a$ be a Lie group corresponding to the PBW degenerate Lie algebra $\g^a$. Let us briefly  describe the Lie group $\G^a$. Let $\mathbb{G}_a$ be the additive group of the field $\mathbb{C}$ and let $M= \dim \n^-$. The Lie group $\G^a$ is a semidirect product  $\mathbb{G}_a^M \rtimes \borel$ of the normal subgroup $\mathbb{G}_a^M$ and the Borel subgroup $\borel$. For any dominant, integral weight $\lambda$, there exist induced $\g^a$- and $\G^a$-module structures on $\V_{\lambda}^a$. 

The group $\G^a$ therefore acts on $\mathbb{P}(\V_{\lambda}^a)$, the projectivization of $\V_{\lambda}^a$. The \emph{PBW degenerate flag variety} (\cite{Fei12}) is defined to be the closure of the $\G^a$-orbit through the highest weight line, that is to say:
$$\F^a := \overline{\G^a [\hv^a]} \hookrightarrow \mathbb{P}(\V_{\lambda}^a).$$

In particular, for $\lambda=\omega_k$, we have the case of Grassmann varieties. For $\g= \mathfrak{sl}_{n+1}(\mathbb{C})$ and $\G=\textsc{SL}_{n+1}(\mathbb{C})$ (type {\tt A}$_n$), we have $\mathcal{F}_{\omega_k}\simeq \mathcal{F}_{\omega_k}^a$, for all $k=1,\ldots,n$. This is true because all the fundamental weights $\omega_k$ are co-minuscule in type {\tt A}, and hence the radical corresponding to each $\omega_k$ is abelian. So, the PBW degenerate flag variety in type {\tt A} is embedded into the product of Grassmannians (\cite[Proposition 3.3]{Fei12}).\\

For $\g=\syp(\mathbb{C})$ and $\G=\group(\mathbb{C})$, we denote by $\Sp\mathcal{F}_{\omega_k}$ and $\Sp\mathcal{F}_{\omega_k}^a$ the symplectic original and PBW degenerate flag varieties. It is known that except for the weight $\omega_n$, all the other fundamental weights of $\syp$ are not co-minuscule. This implies that the radicals corresponding to the weights $\omega_1,\ldots,\omega_{n-1}$ are not abelian.
So, the varieties $\Sp\mathcal{F}_{\omega_k}$ and $\Sp\mathcal{F}_{\omega_k}^a$ are not isomorphic in general, except for the case $k=n$.\\

The variety $\Sp\mathcal{F}_{\omega_k}^a$ is the \emph{symplectic PBW degenerate Grassmannian}, which we will henceforth denote by $\grass^a(k,2n)$, to match the notation we have used for the original symplectic Grassmannian $\grass(k,2n)$. For a dominant, integral and regular weight $\lambda$, the \emph{complete symplectic PBW degenerate flag variety} is the variety:
$$\complete^a := \overline{\group^a [\hv^a]} \hookrightarrow \mathbb{P}(\V_{\lambda}^a).$$
Recall that this is the PBW degeneration of the projective variety $\complete$ from the previous section.  For the rest of the paper, we set our focus on the variety $\complete^a$.

\subsection{The Complete Symplectic PBW Degenerate Flag Variety} The material stated in this subsection is based on \cite{FFL14}. Recall the vector space $\W$ with the fixed basis $\{\w_1,\ldots,\w_{2n}\}$. To begin with, let us recall an important result about the linear algebra realization of the symplectic PBW degenerate Grassmannian $\grass^a(k,2n)$.

Let 
$\W= \W_{k,1}\oplus\W_{k,2}\oplus\W_{k,3},$ where $\W_{k,1}=\text{span}(\w_1,\ldots,\w_k)$, $\W_{k,2}=\text{span}(\w_{k+1},\ldots,\w_{2n-k})$ and $\W_{k,3}=\text{span}(\w_{2n-k+1},\ldots,\w_{2n})$. Let $\operatorname{pr}_{1,3}$ denote the projection $\operatorname{pr}_{1,3}: \W \rightarrow \W_{k,1}\oplus \W_{k,3}$, that is to say, $$\operatorname{pr}_{1,3}(x_1,\ldots,x_{2n})=(x_1,\ldots,x_k,0,\ldots,0,x_{2n-k+1},\ldots,x_{2n}).$$ 
One has:

\begin{proposition}[\cite{FFL14}, Proposition 4.1]
The symplectic PBW degenerate Grassmannian $\grass^a(k,2n)$ is given by:
$$\grass^a(k,2n)= \{\vectorspace_k\in \Gr(k,2n)\,\, | \,\, \operatorname{pr}_{1,3}(\vectorspace_k) \,\, \text{is isotropic}\}.$$
\end{proposition}

Denote by $\operatorname{pr}_k: \W\rightarrow \W$ the projections along $\w_k$, i.e., $$\operatorname{pr}_k(\sum_{j=1}^{2n}c_j \w_j)=\sum_{j\neq k}c_j \w_j.$$ 

The following theorem gives an embedding and a linear algebra realization of the complete symplectic PBW degenerate flag variety $\complete^a$.

\begin{theorem}[\cite{FFL14}, Theorem 4.6]
The variety $\complete^a$ is naturally embedded into the product $\prod_{k=1}^{n}\grass^a(k,2n)$ of symplectic PBW degenerate Grassmannians. The image of this embedding is equal to the projective sub-variety formed by the collections $(\vectorspace_k)_{k=1}^n,$ $\vectorspace_k\in \grass^a(k,2n)$ satisfying the conditions: 
 $\operatorname{pr}_{k+1}\vectorspace_k \subset \vectorspace_{k+1},\,\, k=1,\ldots,n-1$.
\end{theorem}

Consider therefore the embeddings:
\begin{equation}\label{eqn:degenerate-embedding}
    \complete^a \xhookrightarrow{} \prod_{k=1}^{n}\grass^a(k,2n) \xhookrightarrow{} \prod_{k=1}^{n}\Gr(k,2n) \xhookrightarrow{} \prod_{k=1}^n\mathbb{P}\Big(\bigwedge^k \mathbb{C}^{2n}\Big).
\end{equation}

The Pl\"ucker embedding of $\complete^a$ which we consider is the composition of the above embeddings. Our goal is to find the defining ideal of $\complete^a$ with respect to this embedding.

\subsection{Degenerate Relations} One has two kinds of degenerate relations: the linear ones and quadratic ones. These relations live in the polynomial ring $\mathbb{C}[\X_{j_1,\ldots,j_d}^a]$ in variables $\X_{j_1,\ldots,j_d}^a$, $d=1,\ldots,n$ and $1\leq j_1<\cdots<j_d\leq \overline{1}$.

\begin{definition}\label{def:symplectic-degenerate-relations} Recall the notation from Definition \ref{def:classicalrelations}. The \emph{degenerate Pl\"ucker relations} are the relations obtained from the relations $R_{\Lm,\Jm}^t$ by picking up the lowest PBW-degree terms and introducing a superscript ``a" . Therefore we have
\begin{align}\label{eqn:relations-Cn-degenerate}
    R_{\Lm,\Jm}^{t;a}:=\X_{\Lm}^a \X_{\Jm}^a- \sum_{1\leq r_1<\cdots<r_t\leq p}\X_{\Lm'}^a\X_{\Jm'}^a,
\end{align}
labelled by the numbers $p, q$ with $1\leq q\leq p\leq n$, by an integer $t$, $1\leq t\leq q$, and by sequences $\Lm=(l_1,\ldots,l_p)$, $\Jm=(j_1,\ldots,j_q)$ which are subsets of $\{1,\ldots,n,\overline{n},\ldots,\overline{1}\}$.

The \emph{symplectic degenerate relations} are the relations
\begin{equation}\label{eqn:degenerate-linear}
  S_{(\I_2,\I_1)}^a := \X_{(\I_2,\I_1)}^a- \sum_{\Gamma':\, \Gamma'\cap\{\I_1\cup\I_2\}=\emptyset\,\, \text{and}\,\, |\Gamma'|=|\tilde{\Gamma}|}(-1)^{|\Gamma'|} \X_{(\tilde{\I}_2\cup\mathsf{F}\cup \Gamma',\tilde{\I}_1\cup\mathsf{F}\cup \Gamma')}^a,
\end{equation}
where the terms are obtained by picking up the minimum PBW-degree terms from the relations $S_{(\I_2,\I_1)}$ in \eqref{eqn:linear-relations} and introducing a superscript ``a".
\end{definition}

\begin{example}
For $k=n$, we have $S_{(\I_2,\I_1)}=S_{(\I_2,\I_1)}^{a}$ up to a superscript, since the relations $S_{(\I_2,\I_1)}$ are homogeneous with respect to the PBW-degree in this case. This is exactly because we have the isomorphism $\grass(n,2n)\simeq \grass^a(n,2n)$.
\end{example}

\begin{example}
For $n=2$, the degenerate relations for $\comp_4^a$ are:
$$
\begin{array}{rcl}
R_{(1,2),(\overline{2})}^{1;a} :=\X_{1,2}^a\X_{\overline{2}}^a+\X_{2,\overline{2}}^a\X_1^a, & &R_{(1,2),(\overline{2},\overline{1})}^{1;a} :=\X_{1,2}^a\X_{\overline{2},\overline{1}}^a-\X_{1,\overline{2}}^a\X_{2,\overline{1}}^a+\X_{1,\overline{1}}^a\X_{2,\overline{2}}^a,\\
R_{(1,2),(\overline{1})}^{1;a} :=\X_{1,2}^a\X_{\overline{1}}^a+\X_{2,\overline{1}}^a\X_1^a,& & R_{(1,\overline{2}),(\overline{1})}^{1;a} :=\X_{1,\overline{2}}^a\X_{\overline{1}}^a+\X_{\overline{2},\overline{1}}^a\X_1^a-\X_{1,\overline{1}}^a\X_{\overline{2}}^a,\\
R_{(2,\overline{2}),(\overline{1})}^{1;a} :=\X_{2,\overline{2}}^a\X_{\overline{1}}^a-\X_{2,\overline{1}}^a\X_{\overline{2}}^a,&& S_{(1,\overline{1})}^a :=\X_{1,\overline{1}}^a+\X_{2,\overline{2}}^a.
\end{array}
$$
\end{example}

\begin{remark}
To illustrate why the relations  $S_{\Lm}^a$ are obtained by picking up terms of minimal PBW-degree as in Definition \ref{def:symplectic-degenerate-relations}, consider a subspace $\vectorspace_2\subset \mathbb{C}^{6}$, generated by the vectors $u=a_{1,1}\w_1+a_{2,1}\w_2+a_{3,1}\w_3+a_{4,1}\w_4+a_{5,1}\w_5+a_{6,1}\w_6$ and $v=a_{1,2}\w_1+a_{2,2}\w_2+a_{3,2}\w_3+a_{4,2}\w_4+a_{5,2}\w_5+a_{6,2}\w_6$. We want to describe the criterion for $\vectorspace_2$ to be in $\grass^a(2,6)$. Recall the projection $\operatorname{pr}_{1,3}$. Applying it to $u$ and $v$ we have: 
\[\operatorname{pr}_{1,3}(u)=a_{1,1}\w_1+a_{2,1}\w_2+a_{5,1}\w_5+a_{6,1}\w_6,\]
\[\operatorname{pr}_{1,3}(v)=a_{1,2}\w_1+a_{2,2}\w_2+a_{5,2}\w_5+a_{6,2}\w_6.\]
 
\noindent 
Then $\operatorname{pr}_{1,3}(\vectorspace_2)$ is isotropic if and only if $\operatorname{pr}_{1,3}(u)^T \Psi \operatorname{pr}_{1,3}(v)=0$, i.e.,
\[ -a_{6,1}a_{1,2}-a_{5,1}a_{2,2}+a_{2,1}a_{5,2}+a_{1,1}a_{6,2}=0.\]
This leads to the symplectic degenerate relation $\X_{1,6}^a+\X_{2,5}^a=0$ (or 
$\X_{1,\overline{1}}^a+\X_{2,\overline{2}}^a=0$), which is the relation $S_{(1,\overline{1})}^a$ obtained by picking up the terms of minimal PBW-degree from the corresponding relation $S_{(1,\overline{1})}:= \X_{1,\overline{1}}+\X_{2,\overline{2}}+\X_{3,\overline{3}}$.
\end{remark}

\begin{remark}
Observe that the degenerate relations $R_{\Lm,\Jm}^{t;a}$ and $S_{(\I_2,\I_1)}^{a}$ are homogeneous with respect to the PBW-degree. This follows directly from Definition \ref{def:symplectic-degenerate-relations} since the terms in $R_{\Lm,\Jm}^{t;a}$ and $S_{(\I_2,\I_1)}^{a}$ are those of minimal PBW-degrees picked from the relations $R_{\Lm,\Jm}^t$ and $S_{(\I_2,\I_1)}$ respectively.
\end{remark}

 \subsection{A Basis for the Coordinate Ring} Let $\mathbb{C}[\complete^a]$ denote the multi-homogeneous coordinate ring of the complete symplectic PBW degenerate flag variety. 
 Then one has 
 $$\mathbb{C}[\complete^a]=\bigoplus_{\lambda\in P^+}\mathbb{C}[\complete^a]_{\lambda}\simeq\bigoplus_{\lambda\in P^+}(\V_{\lambda}^a)^*,$$
 where the  multiplication  $(\V_{\lambda}^a)^* \otimes (\V_{\mu}^a)^* \rightarrow{} (\V_{\lambda+\mu}^a)^*$ for any two dominant integral weights $\lambda$ and $\mu$ in the algebra on the right hand side is induced by  the embedding of $\g^a$-modules, $\V_{\lambda + \mu}^a\hookrightarrow \V_{\lambda}^a \otimes \V_{\mu}^a$ (according to Lemma \ref{lem:module-embeddings}). 
 Recall that this embedding is compatible with the PBW-degree and hence the algebra $\bigoplus_{\lambda\in P^+}(\V_{\lambda}^a)^*$ is PBW-graded.
 
 The isomorphism of PBW-graded algebras $$\mathbb{C}[\complete^a] \simeq \bigoplus_{\lambda\in P^+}(\V_{\lambda}^a)^*$$ follows from \cite[Theorem 1.2, Theorem 1.4]{FFL14}, which respectively give the flatness of the degeneration and an analogue of the Borel-Weil theorem for the PBW degenerate module $\V_{\lambda}^a$. \\

 We have the elements $\X_{j_1,\ldots,j_k}^a\in (\bigwedge^k \mathbb{C}^{2n})^*$ (Remark \ref{rem:pbwdegree}). By restricting to $\V^a_{\omega_k} \subset \bigwedge^k \mathbb{C}^{2n}$, we can consider these elements also in $(\V^a_{\omega_k})^*$.

\begin{proposition}\label{pro:symplectic-relations}
The degenerate relations $R_{\Lm,\Jm}^{t;a}$ and $S_{(\I_2,\I_1)}^{a}$ are both zero in the coordinate ring $\mathbb{C}[\complete^a]$ with respect to the composition of the embeddings in \eqref{eqn:degenerate-embedding}.
\end{proposition}

\begin{proof}
We know that the elements $\X_{j_1,\ldots,j_k}^a\in (\bigwedge^k \mathbb{C}^{2n})^*$ satisfy relations $R_{\Lm,\Jm}^{t;a}$ in $\mathbb{C}[\complete^a]$ with respect to \eqref{eqn:degenerate-embedding} (see \cite[Lemma 4.1, Theorem 4.10]{Fei12}). 
From Theorem \ref{thm:classical-ideal}, we have that the elements $\X_{j_1,\ldots,j_k}$ satisfy relations $S_{(\I_2,\I_1)}$ in $\mathbb{C}[\complete]$ with respect to the embeddings in \eqref{eqn:classical-embedding}. 

The coordinate ring $\mathbb{C}[\complete^a]$ is PBW-graded and $\complete^a$ is a flat degeneration of $\complete$ (again \cite[Theorem 1.2]{FFL14}), hence the lowest PBW-degree term relations from $S_{(\I_2,\I_1)}$ are equal to $0$. By definition, these lowest term relations are the relations $S_{(\I_2,\I_1)}^{a}$. 
\end{proof}
 
Recall $\SyST_{\lambda}$, the set of all symplectic PBW-semistandard tableaux of shape $\lambda=(\lambda_1\geq \cdots\geq \lambda_n\geq 0)$. To each $\T\in$ $\SyST_{\lambda}$, we associate the monomial element
\[\X_{\T}^a=\prod_{j=1}^{\lambda_1}\X_{\T_{1,j},\ldots,\T_{\mu_j,j}}^a\in (\V_{\lambda}^a)^*,\] 
and call such an element, the \textit{symplectic PBW-semistandard monomial}. We prove the following result.
\begin{theorem} \label{thm:final-result-two}
The elements $\X_{\T}^a$, $\T\in$ $\SyST_{\lambda}$, form a basis of $\mathbb{C}[\complete^a]_{\lambda}$.
\end{theorem}
\begin{proof}
From Theorem \ref{thm:FFLV-degenerate} and Theorem \ref{thm:correspondence}, we know that 
\[\dim \V_{\lambda}^a= \#\{\T: \T\in \SyST_{\lambda}\}.\] It therefore remains to prove that the elements $\X_{\T}^a$, $\T\in$ $\SyST_{\lambda}$ span $\mathbb{C}[\complete^a]_{\lambda}$. From Proposition \ref{pro:symplectic-relations}, we know that the elements $\X_{j_1,\ldots,j_d}^a$ satisfy relations $R_{\Lm,\Jm}^{t;a}$ and $S_{(\I_2,\I_1)}^{a}$ in $\mathbb{C}[\complete^a]$. We are going to use these relations to write an element $\X_{\T}^a$ for which $\T$ is not symplectic PBW-semistandard as a linear combination of elements $\X_{\T'}^a$ with $\T'$ symplectic PBW-semistandard. 

For this, we first follow \cite{Fei12} to define an order on the set of PBW tableaux of shape $\lambda$. Say that $\T^{(1)}>\T^{(2)}$ if there exist $i_0,j_0$ such that $$\T_{i_0,j_0}^{(1)}>\T_{i_0,j_0}^{(2)}\quad \text{and} \quad \T_{i,j}^{(1)}=\T_{i,j}^{(2)}\quad \text{if}\quad (j>j_0, i=i_0) \quad \text{or} \quad (j=j_0, i>i_0).$$ 

Since the condition of PBW-semistandardness is defined between every two adjacent columns, we can reduce the proof to any two such columns. Supposing we are given two such columns $\Lm$ and $\Jm$ that form a PBW tableau that is not PBW-semistandard. We are going to use the degenerate Pl\"ucker relations $R_{\Lm,\Jm}^{t;a}$ to obtain terms corresponding to smaller PBW tableaux. In fact, let $\Lm = (l_1,\ldots,l_p)$ and $\Jm=(j_1,\ldots,j_q)$ with $p\geq q$. From the proof of Proposition 4.12 of \cite{Fei12}, we have that the term $$\X_{l_1,\ldots,l_p}^a\X_{j_1,\ldots,j_q}^a$$ is present in the relation $$R_{(l_1,\ldots,l_p),(j_1,\ldots,j_q)}^{t;a}$$ and that all the other terms correspond to smaller PBW tableaux with respect to the order $``>"$ on the set of PBW tableaux. 

The only thing that remains is to show that we can write a term corresponding to a PBW-semistandard tableau that is not symplectic as a linear combination of terms corresponding to symplectic 
PBW-semistandard tableaux. For this, we use the symplectic degenerate relations. Indeed, let $\Lm'$ be any non symplectic column that appears as a result of the exchange process during the application of the degenerate Pl\"ucker relations above. Then from Lemma \ref{lem:PBW-linear}, the term $\X_{\Lm'}$ is among the terms with minimal PBW-degree in $S_{\Lm'}$. This means that the term $\X_{\Lm'}^a$ is present in the relation $S_{\Lm'}^a$ since this relation is obtained by picking up terms of minimal PBW-degree from $S_{\Lm'}$. We can therefore use the relation $S_{\Lm'}^a$ to replace terms corresponding to non symplectic columns. 

We claim that the new columns that arise in this process are smaller than $\Lm'$ with respect to the order $``>"$. For this, recall the definition of the relations $S_{\Lm'}$ and $S_{\Lm'}^a$ and the used notation. From Proposition \ref{pro:reversed-admissible-minors} and Lemma \ref{lem:PBW-linear}, we know that the largest element which can be removed from $\Lm'=(\I_2,\I_1)'$ is $\gamma_b$ and that $\gamma_b<k$. Hence we also know that the PBW-degree goes up only when there exists $i$ with $h_0\leq i \leq b$ such that $\gamma_i'>k$, among the new entries. Therefore since we are using relations $S_{\Lm'}^a$, it suffices to consider the case $\gamma_b'\leq k$. For any given term $\X_{\Lm''}^a$ in $S_{\Lm'}^a$ apart from $\X_{\Lm'}^a$, and for the corresponding sequence $\Lm''$, let $f$ be the position of $\gamma_b'$ after rearranging the entries to form a single column symplectic PBW tableau.\\
Clearly, we need to begin comparing the entries of the columns $\Lm'$ and $\Lm''$ starting from position $f$ downwards. To see this, recall that since $\gamma_b'\leq k$, then $f=\gamma_b'$. This implies that the entry at position $f$ in $\Lm'$ is different from $f$ since $\gamma_b'\in \{1,\ldots,n\}\setminus (\I_1\cup\I_2)' $ with $\Lm'=(\I_2,\I_1)'$. Let $\Lm_f'$ denote the entry at position $f$ in single column symplectic PBW tableau $\Lm'$. We have $\Lm_f'>f=\gamma_b'$. Moreover, all entries below position $f$ (if any), are pairwise equal in $\Lm'$ and $\Lm''$. This implies that $\Lm'>\Lm''$, and hence the claim is proved. 
\end{proof}

\subsection{Defining Ideal}
Let $\Ideal^a\subset \mathbb{C}[\X_{j_1,\ldots,j_d}^a]$ be the ideal generated by the degenerate relations  $R_{\Lm,\Jm}^{t;a}$ and $S_{(\I_2,\I_1)}^{a}$. The following is the main statement of this paper.

\begin{theorem}\label{thm:defining-relations-degenerate}
The ideal $\Ideal^a$ is the defining ideal of the variety $\complete^a$ under the Pl\"ucker embedding, $\complete^a\xhookrightarrow{}\prod_{k=1}^{n} \mathbb{P}\Big(\bigwedge^k \mathbb{C}^{2n}\Big)$.
\end{theorem}
\begin{proof}
From Theorem \ref{thm:final-result-two}, we see that the relations $R_{\Lm,\Jm}^{t;a}$ and $S_{(\I_2,\I_1)}^{a}$ in $\Ideal^a$ are enough to express every monomial in Pl\"ucker coordinates as a linear combination of symplectic PBW-semistandard monomials (i.e., these relations provide a straightening law for $\mathbb{C}[\complete^a]$). 

Following the idea of the proof of Theorem 7 in \cite{GFFFR19}, this implies that the ideal $\Ideal^a$ is the defining ideal of $\complete^a$ since otherwise, it would imply that the symplectic PBW-semistandard monomials are not a basis for $\mathbb{C}[\complete^a]$. 
\end{proof}

\begin{remark}
From Theorem \ref{thm:defining-relations-degenerate}, we can now write down the multi-homogeneous coordinate ring of $\complete^a$ as a quotient of the polynomial ring $\mathbb{C}[\X_{j_1,\ldots,j_d}^a]$ by the ideal $\Ideal^a$, i.e.,
$$\mathbb{C}[\complete^a]=\mathbb{C}[\X_{j_1,\ldots,j_d}^a]/\Ideal^a \simeq\bigoplus_{\lambda\in P^+}(\V_{\lambda}^a)^*.$$ 
\end{remark}

\begin{corollary}
The ideal $\Ideal^a$ is a prime ideal of the polynomial ring $\mathbb{C}[\X_{j_1,\ldots,j_d}^a]$.
\end{corollary}
\begin{proof}
This follows directly from Theorem \ref{thm:defining-relations-degenerate} and the fact that the variety $\complete^a$ is irreducible (see \cite[Corollary 5.6]{FFL14}). 
\end{proof}

\begin{remark}
We know from \cite[Theorem 1.2]{FFL14} that $\complete^a$ is a flat degeneration of $\complete$. We would like to give a formulation of this result in terms of the results of this paper. Let $s$ be a variable. We follow \cite{Fei12} to define an algebra $Q^s$ over the ring $\mathbb{C}[s]$ as a quotient of the polynomial ring $\mathbb{C}[s][\X_{j_1,\ldots,j_d}^a], \,\, d=1,\ldots, n$ by the ideal $\Ideal^s$ generated by quadratic relations $R_{\Lm,\Jm}^{t;s}$ and linear relations $S_{(\I_2,\I_1)}^{s}$ which are $s$-deformations of the relations $R_{\Lm,\Jm}^t$ and $S_{(\I_2,\I_1)}$. Let $R_{\Lm,\Jm}^t = \sum_i \X_{\Lm^{(i)}}\X_{\Jm^{(i)}}$ and $S_{(\I_2,\I_1)}=\sum_i \X_{(\I_2,\I_1)^{(i)}}$, then:
\begin{align*}
    R_{\Lm,\Jm}^{t;s} &= s^{-\min_i(\deg \Lm^{(i)} + \deg \Jm^{(i)})}\sum_i s^{\deg \Lm^{(i)} + \deg \Jm^{(i)}} \X_{\Lm^{(i)}}\X_{\Jm^{(i)}},\\
    S_{(\I_2,\I_1)}^{s}&=s^{-\min_i(\deg\, (\I_2,\I_1)^{(i)})}\sum_i s^{\deg \, (\I_2,\I_1)^{(i)}} \X_{(\I_2,\I_1)^{(i)}}.
\end{align*}
We have $Q^s/(s)\simeq \mathbb{C}[\complete^a]$, and $Q^s/(s-u)\simeq \mathbb{C}[\complete]$ for $u\neq 0$. Moreover, following Theorem \ref{thm:final-result-two}, one checks that the elements $\X_{\T}$, $\T\in$ $\SyST_{\lambda}$, $\lambda\in P^+$, form a $\mathbb{C}[s]$ basis of $Q^s$, hence showing that $Q^s$ is free over $\mathbb{C}[s]$.
\end{remark}

\end{document}